\documentclass{amsart}
\usepackage{amsfonts}
\usepackage{amsmath}
\usepackage{amssymb}
\usepackage{amsthm}
\usepackage{graphicx}
\usepackage{capt-of}
\usepackage{tikz}
\usetikzlibrary{patterns}
\usetikzlibrary{decorations.markings}
\usepackage{multirow}
\usepackage{array} 
\usetikzlibrary{calc}

\newtheorem{theorem}{Theorem}

\newtheorem{lemma}{Lemma}
\newtheorem{rem}{Remark}
\newtheorem{exmp}{Example}

\begin{document}
\title{On $Z$-monodromies in embedded graphs}
\author{Adam Tyc}
\subjclass[2000]{}

\
\address{Adam Tyc: Faculty of Mathematics and Computer Science, 
University of Warmia and Mazury, S{\l}oneczna 54, 10-710 Olsztyn, Poland}
\email{adam.tyc@matman.uwm.edu.pl}

\maketitle
\begin{abstract}
We characterize all permutations which realize as the $z$-monodro\-mies of faces in connected simple finite graphs embedded in surfaces whose duals are also simple. 
\end{abstract}

\section{Introduction}
{\it Zigzags} are closed walks in embedded graphs which generalize the concept of {\it Petrie polygons} in regular polytopes. 
They were used in computer graphics \cite{TriApp} and in enumerating all combinatorial possibilities for fullerenes in mathematical chemistry \cite{BD,DDS-book}. 
Zigzags are also closely related to {\it Gauss code problem}: if an embedded graph contains a single zigzag, then this zigzag is a geometrical realization of a certain Gauss code (see \cite[Section 17.7]{GR-book} for the planar case and \cite{CrRos,Lins2} for the case when a graph is embedded in an arbitrary surface).
More results on zigzags can be found in \cite{Lins1, PT1, Shank, T1}.

We will consider zigzags in connected simple finite graphs embedded in surfaces whose duals are also simple. 
The latter condition guarantees that for any two consecutive edges on a face there is a unique zigzag containing them. 
This property is the crucial tool in the concept of {\it $z$-monodromy}. 
For a face $F$ the $z$-monodromy $M_F$ is a permutation on the set of all oriented edges of $F$. 
If $e_0,e$ is a pair of consecutive edges in $F$, then $M_F(e)$ is the first oriented edge of $F$ that occurs in the zigzag containing $e_0,e$ after $e$.  
So, $z$-monodromies are transformations of the first recurrence of zigzags to faces.  

Such $z$-monodromies were introduced in \cite{PT2} and exploited to prove that any triangulation of an arbitrary (not necessarily oriented) closed surface can be shredded to a triangulation with a single zigzag. 
There are precisely $7$ types of $z$-monodromies for triangle faces and each of them is realized. 
The properties and some applications of $z$-monodromies of triangle faces can be found in \cite{PT3,T3}. 
See also \cite{T2} for a generalization of $z$-monodromies on pairs of edges. 

Faces of embedded graph under consideration contains at least three edges. 
We characterize permutations that realize as $z$-monodromies of $k$-gonal faces for any $k\geq 3$. 
More precisely, a permutation $\sigma$ on the set 
$$[k]_{\pm}=\{1,\dots,k,-k,\dots,-1\}$$ 
realizes as the $z$-monodromy if and only if it satisfies the following conditions:
\begin{enumerate}
\item[$\bullet$] if $\sigma(i)=j$, then $\sigma(-j)=-i$;
\item[$\bullet$] $\sigma(i)\neq-i$.
\end{enumerate} 
In the plane case, our construction is based on the {\it chess coloring} of $4$-regular plane graphs. 
Using the connected sum of surfaces, we extend the result on the non-planar case. 

We consider the case when an embedded graph and its dual both are simple. 
In the general case, zigzags cannot be reconstructed from pairs of consecutive edges. 
This shows that the concept of $z$-monodromy cannot be generalized in a direct way. 

\section{Zigzags in embedded graphs}

Let $S$ be a connected closed $2$-dimensional (not necessarily orientable) surface. 
Let $\Gamma$ be a $2$-cell embedding of a connected finite graph in $S$, in other words, a {\it map} \cite[Definition 1.3.6]{LZsurf}. 
The difference $S\setminus \Gamma$ is a disjoint union of open disks and the closures of  these disks are the {\it faces}. 
We say that a face is {\it $k$-gonal} if it contains precisely $k$ edges. 
We will always assume that the following condition is satisfied:
\begin{enumerate}
\item[(SS)] $\Gamma$ and the dual map $\Gamma^*$ (see \cite[p.52]{LZsurf})  both are embeddings of simple graphs. 
\end{enumerate}
The fact that one of the graphs is simple does not implies that the same holds for the other graph. 
For example, $\Gamma^*$ is not simple if $\Gamma$ contains a vertex of degree $2$ or two distinct faces with intersection containing more than one edge. 
The condition (SS) implies that each face in our graphs is $k$-gonal with $k\geq 3$.  

A {\it zigzag}  in $\Gamma$ is a {\it sequence} of edges $\{e_{i}\}_{i\in {\mathbb N}}$ satisfying the following conditions for every $i\in {\mathbb N}$: 
\begin{enumerate}
\item[$\bullet$] $e_{i}$ and $e_{i+1}$ are distinct, they have a common vertex and belong to the same face, 
\item[$\bullet$] the faces containing $e_{i},e_{i+1}$ and $e_{i+1},e_{i+2}$ are distinct
and the edges $e_{i}$ and $e_{i+2}$ are non-intersecting.
\end{enumerate} 
Since $\Gamma$ is finite, 
there is a natural number $n>0$ such that $e_{i+n}=e_{i}$ for every natural $i$. 
Thus, every zigzag will be represented as a cyclic sequence $e_{1},\dots,e_{n}$,
where $n$ is the smallest number satisfying this condition.

Any zigzag is completely determined by every pair of consecutive edges contained in this zigzag. 
Conversely, for every pair of distinct edges $e, e'$ which have a common vertex and belong to the same face there is a unique zigzag containing the sequence $e, e'$.
This property will be used in the next section. 

If $Z=\{e_{1},\dots,e_{n}\}$ is a zigzag, then the reversed sequence $Z^{-1}=\{e_{n},\dots,e_{1}\}$ also is a zigzag.
A zigzag cannot contain a sequence $e,e',\dots,e',e$ which implies that
$Z\ne Z^{-1}$ for any zigzag $Z$. 
In other words, a zigzag cannot be self-reversed (see \cite{PT2} for the proof for triangulations; in our case the proof is similar).

\begin{exmp}{\rm
Consider the cube $Q_3$ whose vertices are $1,\dots,8$, see Fig. 1. 
\begin{center}
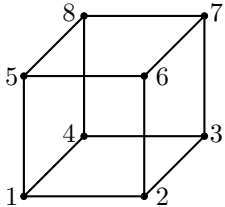

\begin{tikzpicture}[scale=0.8]


\draw[fill=black] (0,0) circle (1.5pt);
\draw[fill=black] (2,0) circle (1.5pt);
\draw[fill=black] (1,1) circle (1.5pt);
\draw[fill=black] (3,1) circle (1.5pt);

\draw[fill=black] (0,2) circle (1.5pt);
\draw[fill=black] (2,2) circle (1.5pt);
\draw[fill=black] (1,3) circle (1.5pt);
\draw[fill=black] (3,3) circle (1.5pt);

\draw[thick] (0,0)--(2,0)--(3,1)--(1,1)--cycle;
\draw[thick] (0,2)--(2,2)--(3,3)--(1,3)--cycle;

\draw[thick] (0,0)--(0,2);
\draw[thick] (2,0)--(2,2);
\draw[thick] (3,1)--(3,3);
\draw[thick] (1,1)--(1,3);

\node at (-0.2,0) {$1$};
\node at (2.3,0) {$2$};
\node at (0.75,1.05) {$4$};
\node at (3.2,1.05) {$3$};

\node at (-0.2,2) {$5$};
\node at (2.3,2) {$6$};
\node at (0.75,3.05) {$8$};
\node at (3.2,3.05) {$7$};

\end{tikzpicture}
\captionof{figure}{The cube $Q_3$}
\end{center}
It contains precisely $4$ zigzags up to reversing: 
$$12,23,37,78,85,51;\hspace{0.22cm}12,26,67,78,84,41;\hspace{0.22cm}14,43,37,76,65,51;\hspace{0.22cm}23,34,48,85,56,62.$$
Let $BP_n$ be  the $n$-gonal bipyramid, where $1,\dots,n$ are the consecutive vertices of the base 
and the remaining two vertices are $a,b$. 
\begin{center}
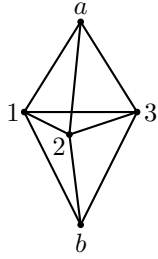

\begin{tikzpicture}[scale=0.3]


\coordinate (A) at (0.5,4);
\coordinate (B) at (0.5,-5);
\coordinate (1) at (0,-1);
\coordinate (2) at (-2,0);
\coordinate (3) at (3,0);

\draw[fill=black] (A) circle (3.5pt);
\draw[fill=black] (B) circle (3.5pt);

\draw[fill=black] (1) circle (3.5pt);
\draw[fill=black] (2) circle (3.5pt);
\draw[fill=black] (3) circle (3.5pt);

\draw[thick] (3)--(1)--(2);
\draw[thick] (2)--(3);

\draw[thick] (A)--(1);
\draw[thick] (A)--(2);
\draw[thick] (A)--(3);

\draw[thick] (B)--(1);
\draw[thick] (B)--(2);
\draw[thick] (B)--(3);

\node at (3.6,0) {$3$};
\node at (-0.45,-1.525) {$2$};
\node at (-2.5,0) {$1$};

\node at (0.5,4.6) {$a$};
\node at (0.5,-5.8) {$b$};

\node[color=white] at (4,0) {$.$};

\end{tikzpicture}
\captionof{figure}{The bipyramid $BP_3$}
\end{center}
If $n=3$ (see Fig. 2), then it contains a single zigzag (up to reversing): 
$$a1,12,2b,b3,31,1a,a2,23,3b,b1,12,2a,a3,31,1b,b2,23,3a.$$
The same holds for $BP_n$ if $n$ is odd. 
If $n$ is even, then $BP_n$ contains $2$ or $4$ zigzags up to reversing. 
}\end{exmp}

Every zigzag in $\Gamma$ induces in a natural way a zigzag in $\Gamma^*$ and vice versa. 
\begin{rem}{\rm
Zigzags can be defined in maps of non-simple graphs \cite{Lins1}. 
In this case, there are simple examples showing that a zigzag cannot be determined by any pair of its consecutive edges. 
}\end{rem}

\section{Main result}
Let $\Gamma$ be as in the previous section and let $F$ be a $k$-gonal face of $\Gamma$.  
Denote by $v_0,\dots,v_{k-1}$ the consecutive vertices of $F$ in a fixed orientation on the boundary of this face (it is possible that $v_i=v_j$ if $|i-j|\geq3$). 
Consider the set of all oriented edges of $F$
$$\Omega(F)=\{e_1,\dots,e_k,-e_k,\dots,-e_1\},$$
where $e_i=v_{i-1}v_i$ and $-e_i=v_iv_{i-1}$ are mutually reversed oriented edges of $F$ (the indices are taken modulo $k$); 
it is clear that $\Omega(F)$ consists of $2k$ mutually distinct elements. 
Let $D_F$ be the following permutation on $\Omega(F)$
$$D_F=(e_1,e_2,\dots,e_k)(-e_k,\dots,-e_2,-e_1).$$
In other words, $D_F$ transfers every oriented edge of $F$ to the next oriented edge in the corresponding orientation on the boundary. 

The {\it $z$-monodromy} of $F$ is the transformation $M_F$ on $\Omega(F)$ defined as follows. 
For any $e\in\Omega(F)$ we take $e_0\in\Omega(F)$ such that $D_F(e_0)=e$. 
There is a unique zigzag, where $e_0,e$ are consecutive edges. 
The first element of $\Omega(F)$ contained in this zigzag after $e_0,e$  is denoted by $M_F(e)$. 
\begin{rem}{\rm
The $z$-monodromy is defined when (SS) is satisfied. This concept cannot be carried out on the general case immediately. 
}\end{rem}
\begin{lemma}\label{lem1}
The following assertions are fulfilled: 
\begin{enumerate}
\item[(1)] If $M_{F}(e)=e'$ for some $e,e'\in \Omega(F)$, then $M_{F}(-e')=-e$. 
\item[(2)] $M_{F}$ is bijective.
\item[(3)] $M_{F}(e)\ne -e$ for every $e\in \Omega(F)$.
\end{enumerate}
\end{lemma}
\begin{proof}
(1). 
Let $e\in \Omega(F)$. Consider $e_{0}\in\Omega(F)$ satisfying $D_{F}(e_{0})=e$.
If $Z$ is the zigzag containing the pair $e_{0},e$, 
then 
$$e'=M_{F}(e)\;\mbox{ and }\;e'_{0}=D_{F}M_{F}(e)$$ 
are the next two elements of $\Omega(F)$ in $Z$.
Observe that $D_{F}(-e'_{0})=-e'$.
The reversed zigzag $Z^{-1}$ contains the sequence $-e'_{0},-e'$ and $-e$ is the first element of $\Omega(F)$ contained in $Z^{-1}$ after this pair. 
This means that $M_{F}(-e')=-e$.

(2).
It is sufficient to show that $M_{F}$ is injective.
Suppose that $M_{F}(e)=M_{F}(e')=e''$.
By (1), we have $-e=M_{F}(-e'')=-e'$ which implies that $e=e'$.

(3). 
Let $e$ and $e_{0}$ be as in the proof of (1).
If $M_{F}(e)=-e$, then there is a zigzag $Z$ containing the sequences $e_{0},e$ and $-e,D_{F}(-e)$.
Since $D_{F}(-e)=-e_{0}$, $Z$ passes through both pairs $e_{0},e$  and $-e,-e_{0}$. 
This implies that $Z=Z^{-1}$ which is impossible. 
\end{proof}

The set $\Omega(F)$ is naturally identified with
$$[k]_{\pm}=[k]_{+}\cup[k]_{-}$$
where 
$$[k]_{+}=\{1,\dots,k\},\text{ and }[k]_{-}=\{-k,\dots,-1\}$$
($e_i$ and $-e_i$ correspond to $i$ and $-i$, respectively). 
Then, by Lemma \ref{lem1}, the $z$-monodromy $M_F$ is a permutation $\sigma$ of $[k]_{\pm}$ satisfying the following conditions:
\begin{enumerate}
\item[(M1)] if $\sigma(i)=j$, then $\sigma(-j)=-i$;
\item[(M2)] $\sigma(i)\neq-i$.
\end{enumerate}
\noindent Our main result is the following. 
\begin{theorem}\label{th1}
Let $S$ be a connected closed $2$-dimensional (not necessarily orientable) surface. 
Then every permutation $\sigma$ of $[k]_{\pm}$ satisfying {\rm (M1)} and {\rm (M2)} is realized as the $z$-monodromy of a $k$-gonal face in a connected finite graph $\Gamma$ embedded in $S$ and satisfying {\rm (SS)}. 
\end{theorem}

In Section 4, we prove Theorem \ref{th1} for plane graphs (graphs embedded in a sphere). 
Graphs on surfaces different from a sphere will be considered in Section 5. 

\section{The plane case}
\subsection{Preliminary}
Let $G$ be a $4$-regular plane graph.
The dual graph $G^*$ is bipartite and there exists a {\it chess coloring} of faces of $G$ in two colors $b$ and $w$. 
For $c\in\{b,w\}$ we take a vertex inside every face of $G$ colored in $c$ and join two such vertices by an edge if the corresponding faces have a  common vertex at theirs boundaries. 
The obtained plane graph will be denoted by $\mathcal{R}_c(G)$. 
The graphs $\mathcal{R}_b(G)$ and $\mathcal{R}_{w}(G)$ are dual (see Fig. 3). 
\begin{center}
\begin{tikzpicture} 
\begin{scope}[scale=0.85]

\fill [opacity=0.2,gray] (0:3cm)--(30:2cm)--(60:3cm) -- cycle;
\fill [opacity=0.2,gray] (120:3cm)--(90:2cm)--(60:3cm) -- cycle;
\fill [opacity=0.2,gray] (120:3cm)--(150:2cm)--(180:3cm) -- cycle;
\fill [opacity=0.2,gray] (240:3cm)--(210:2cm)--(180:3cm) -- cycle;
\fill [opacity=0.2,gray] (240:3cm)--(270:2cm)--(300:3cm) -- cycle;
\fill [opacity=0.2,gray] (0:3cm)--(330:2cm)--(300:3cm) -- cycle;

\fill [opacity=0.2,gray] (30:2cm)--(0.85cm,0.5cm)--(1.15cm,0cm) -- cycle;
\fill [opacity=0.2,gray] (90:2cm)--(-0.87cm,0.5cm)--(0,0)--(0.85cm,0.5cm) -- cycle;
\fill [opacity=0.2,gray] (150:2cm)--(-0.87cm,0.5cm)--(-1.15cm,0cm) -- cycle;
\fill [opacity=0.2,gray] (210:2cm)--(-1.15cm,0cm)--(-0.86cm,-0.5cm) -- cycle;
\fill [opacity=0.2,gray] (270:2cm)--(-0.86cm,-0.5cm)--(0,0)-- (0.86cm,-0.5cm) -- cycle;
\fill [opacity=0.2,gray] (330:2cm)--(0.86cm,-0.5cm)--(1.15cm,0cm) -- cycle;

\draw[fill=black] (0:3cm) circle (2pt);
\draw[fill=black] (60:3cm) circle (2pt);
\draw[fill=black] (120:3cm) circle (2pt);
\draw[fill=black] (180:3cm) circle (2pt);
\draw[fill=black] (240:3cm) circle (2pt);
\draw[fill=black] (300:3cm) circle (2pt);
    \draw[thick] (0:3cm) \foreach \x in {60, 120,...,359} {
            -- (\x:3cm) 
        } -- cycle (60:3cm);

\draw[fill=black] (30:2cm) circle (2pt);
\draw[fill=black] (90:2cm) circle (2pt);
\draw[fill=black] (150:2cm) circle (2pt);
\draw[fill=black] (210:2cm) circle (2pt);
\draw[fill=black] (270:2cm) circle (2pt);
\draw[fill=black] (330:2cm) circle (2pt);

    \draw[thick] (0:3cm)--(30:2cm)--(60:3cm);
    \draw[thick] (120:3cm)--(90:2cm)--(60:3cm);
    \draw[thick] (120:3cm)--(150:2cm)--(180:3cm);
    \draw[thick] (240:3cm)--(210:2cm)--(180:3cm);
    \draw[thick] (240:3cm)--(270:2cm)--(300:3cm);
    \draw[thick] (0:3cm)--(330:2cm)--(300:3cm);

    \draw[thick] (30:2cm)--(270:2cm);
    \draw[thick] (30:2cm)--(210:2cm);

    \draw[thick] (90:2cm)--(330:2cm);
    \draw[thick] (90:2cm)--(210:2cm);

\draw[fill=black] (1.15cm,0cm) circle (2pt);
\node (z2) at (1.15cm,0cm) {};
\draw[fill=black] (0.85cm,0.5cm) circle (2pt);
\node (z1) at (0.85cm,0.5cm) {};

    \draw[thick] (150:2cm)--(270:2cm);
    \draw[thick] (150:2cm)--(330:2cm);

\draw[fill=black] (-1.15cm,0cm) circle (2pt);
\node (z3) at (-1.15cm,0cm) {};
\draw[fill=black] (-0.87cm,0.5cm) circle (2pt);
\node (z4) at (-0.87cm,0.5cm) {};
\draw[fill=black] (0,0) circle (2pt);
\node (z5) at (0,0) {};
\draw[fill=black] (-0.86cm,-0.5cm) circle (2pt);
\node (z6) at (-0.86cm,-0.5cm) {};
\draw[fill=black] (0.86cm,-0.5cm) circle (2pt);
\node (z7) at (0.86cm,-0.5cm) {};

\node (x1) at (0:3cm) {};
\node (x2) at (60:3cm) {};
\node (x3) at (120:3cm) {};
\node (x4) at (180:3cm) {};
\node (x5) at (240:3cm) {};
\node (x6) at (300:3cm) {};

\node (y1) at (30:2cm) {};
\node (y2) at (90:2cm) {};
\node (y3) at (150:2cm) {};
\node (y4) at (210:2cm) {};
\node (y5) at (270:2cm) {};
\node (y6) at (330:2cm) {};


\draw[fill=red, color=red] (barycentric cs:x1=1,x2=1,y1=1) circle (2pt);
\draw[fill=red, color=red] (barycentric cs:x2=1,x3=1,y2=1) circle (2pt);
\draw[fill=red, color=red] (barycentric cs:x3=1,x4=1,y3=1) circle (2pt);
\draw[fill=red, color=red] (barycentric cs:x4=1,x5=1,y4=1) circle (2pt);
\draw[fill=red, color=red] (barycentric cs:x5=1,x6=1,y5=1) circle (2pt);
\draw[fill=red, color=red] (barycentric cs:x6=1,x1=1,y6=1) circle (2pt);
    \draw[color=red] (barycentric cs:x1=1,x2=1,y1=1)--(barycentric cs:x2=1,x3=1,y2=1)--(barycentric cs:x3=1,x4=1,y3=1)--(barycentric cs:x4=1,x5=1,y4=1)--(barycentric cs:x5=1,x6=1,y5=1)--(barycentric cs:x6=1,x1=1,y6=1)--cycle;

\draw[fill=red, color=red] (barycentric cs:z1=1,z2=1,y1=1) circle (2pt);
\draw[fill=red, color=red] (barycentric cs:z1=1,z4=1,y2=1,z5=1) circle (2pt);
\draw[fill=red, color=red] (barycentric cs:z4=1,z3=1,y3=1) circle (2pt);
\draw[fill=red, color=red] (barycentric cs:z3=1,z6=1,y4=1) circle (2pt);
\draw[fill=red, color=red] (barycentric cs:z5=1,z6=1,y5=1,z7=1) circle (2pt);
\draw[fill=red, color=red] (barycentric cs:z2=1,z7=1,y6=1) circle (2pt); 
    \draw[color=red] (barycentric cs:z1=1,z2=1,y1=1)--(barycentric cs:z1=1,z4=1,y2=1,z5=1)--(barycentric cs:z4=1,z3=1,y3=1)--(barycentric cs:z3=1,z6=1,y4=1)--(barycentric cs:z5=1,z6=1,y5=1,z7=1)--(barycentric cs:z2=1,z7=1,y6=1)--cycle;
    \draw[color=red] (barycentric cs:z1=1,z4=1,y2=1,z5=1)--(barycentric cs:z5=1,z6=1,y5=1,z7=1);

    \draw[color=red] (barycentric cs:x1=1,x2=1,y1=1)--(barycentric cs:y1=1,z1=1,z2=1);
    \draw[color=red] (barycentric cs:x2=1,x3=1,y2=1)--(barycentric cs:y2=1,z1=1,z4=1,z5=1);
    \draw[color=red] (barycentric cs:x3=1,x4=1,y3=1)--(barycentric cs:y3=1,z3=1,z4=1);
    \draw[color=red] (barycentric cs:x4=1,x5=1,y4=1)--(barycentric cs:y4=1,z3=1,z6=1);
    \draw[color=red] (barycentric cs:x5=1,x6=1,y5=1)--(barycentric cs:y5=1,z6=1,z7=1,z5=1);
    \draw[color=red] (barycentric cs:x6=1,x1=1,y6=1)--(barycentric cs:y6=1,z2=1,z7=1);

\begin{scope}[xshift=7cm]

\fill [opacity=0.2,gray] (0:3cm)--(30:2cm)--(60:3cm) -- cycle;
\fill [opacity=0.2,gray] (120:3cm)--(90:2cm)--(60:3cm) -- cycle;
\fill [opacity=0.2,gray] (120:3cm)--(150:2cm)--(180:3cm) -- cycle;
\fill [opacity=0.2,gray] (240:3cm)--(210:2cm)--(180:3cm) -- cycle;
\fill [opacity=0.2,gray] (240:3cm)--(270:2cm)--(300:3cm) -- cycle;
\fill [opacity=0.2,gray] (0:3cm)--(330:2cm)--(300:3cm) -- cycle;

\fill [opacity=0.2,gray] (30:2cm)--(0.85cm,0.5cm)--(1.15cm,0cm) -- cycle;
\fill [opacity=0.2,gray] (90:2cm)--(-0.87cm,0.5cm)--(0,0)--(0.85cm,0.5cm) -- cycle;
\fill [opacity=0.2,gray] (150:2cm)--(-0.87cm,0.5cm)--(-1.15cm,0cm) -- cycle;
\fill [opacity=0.2,gray] (210:2cm)--(-1.15cm,0cm)--(-0.86cm,-0.5cm) -- cycle;
\fill [opacity=0.2,gray] (270:2cm)--(-0.86cm,-0.5cm)--(0,0)-- (0.86cm,-0.5cm) -- cycle;
\fill [opacity=0.2,gray] (330:2cm)--(0.86cm,-0.5cm)--(1.15cm,0cm) -- cycle;

\draw[fill=black] (0:3cm) circle (2pt);
\draw[fill=black] (60:3cm) circle (2pt);
\draw[fill=black] (120:3cm) circle (2pt);
\draw[fill=black] (180:3cm) circle (2pt);
\draw[fill=black] (240:3cm) circle (2pt);
\draw[fill=black] (300:3cm) circle (2pt);
    \draw[thick] (0:3cm) \foreach \x in {60, 120,...,359} {
            -- (\x:3cm) 
        } -- cycle (60:3cm);

\draw[fill=black] (30:2cm) circle (2pt);
\draw[fill=black] (90:2cm) circle (2pt);
\draw[fill=black] (150:2cm) circle (2pt);
\draw[fill=black] (210:2cm) circle (2pt);
\draw[fill=black] (270:2cm) circle (2pt);
\draw[fill=black] (330:2cm) circle (2pt);

    \draw[thick] (0:3cm)--(30:2cm)--(60:3cm);
    \draw[thick] (120:3cm)--(90:2cm)--(60:3cm);
    \draw[thick] (120:3cm)--(150:2cm)--(180:3cm);
    \draw[thick] (240:3cm)--(210:2cm)--(180:3cm);
    \draw[thick] (240:3cm)--(270:2cm)--(300:3cm);
    \draw[thick] (0:3cm)--(330:2cm)--(300:3cm);

    \draw[thick] (30:2cm)--(270:2cm);
    \draw[thick] (30:2cm)--(210:2cm);

    \draw[thick] (90:2cm)--(330:2cm);
    \draw[thick] (90:2cm)--(210:2cm);

\draw[fill=black] (1.15cm,0cm) circle (2pt);
\node (z2) at (1.15cm,0cm) {};
\draw[fill=black] (0.85cm,0.5cm) circle (2pt);
\node (z1) at (0.85cm,0.5cm) {};

    \draw[thick] (150:2cm)--(270:2cm);
    \draw[thick] (150:2cm)--(330:2cm);

\draw[fill=black] (-1.15cm,0cm) circle (2pt);
\node (z3) at (-1.15cm,0cm) {};
\draw[fill=black] (-0.87cm,0.5cm) circle (2pt);
\node (z4) at (-0.87cm,0.5cm) {};
\draw[fill=black] (0,0) circle (2pt);
\node (z5) at (0,0) {};
\draw[fill=black] (-0.86cm,-0.5cm) circle (2pt);
\node (z6) at (-0.86cm,-0.5cm) {};
\draw[fill=black] (0.86cm,-0.5cm) circle (2pt);
\node (z7) at (0.86cm,-0.5cm) {};

\node (x1) at (0:3cm) {};
\node (x2) at (60:3cm) {};
\node (x3) at (120:3cm) {};
\node (x4) at (180:3cm) {};
\node (x5) at (240:3cm) {};
\node (x6) at (300:3cm) {};

\node (y1) at (30:2cm) {};
\node (y2) at (90:2cm) {};
\node (y3) at (150:2cm) {};
\node (y4) at (210:2cm) {};
\node (y5) at (270:2cm) {};
\node (y6) at (330:2cm) {};


\draw[fill=blue, color=blue] (barycentric cs:x1=1,y1=1,z2=1,y6=1) circle (2pt);
\draw[fill=blue, color=blue] (barycentric cs:x2=1,y2=1,z1=1,y1=1) circle (2pt);
\draw[fill=blue, color=blue] (barycentric cs:x3=1,y3=1,z4=1,y2=1) circle (2pt);
\draw[fill=blue, color=blue] (barycentric cs:x4=1,y4=1,z3=1,y3=1) circle (2pt);
\draw[fill=blue, color=blue] (barycentric cs:x5=1,y5=1,z6=1,y4=1) circle (2pt);
\draw[fill=blue, color=blue] (barycentric cs:x6=1,y5=1,y6=1,z7=1) circle (2pt);

    \draw[color=blue] (barycentric cs:x1=1,y1=1,z2=1,y6=1)--(barycentric cs:x2=1,y2=1,z1=1,y1=1)--(barycentric cs:x3=1,y3=1,z4=1,y2=1)-- (barycentric cs:x4=1,y4=1,z3=1,y3=1) --(barycentric cs:x5=1,y5=1,z6=1,y4=1)--(barycentric cs:x6=1,y5=1,y6=1,z7=1)--cycle;

\draw[fill=blue, color=blue] (barycentric cs:z3=1,z4=1,z5=1,z6=1) circle (2pt);
\draw[fill=blue, color=blue] (barycentric cs:z1=1,z2=1,z7=1,z5=1) circle (2pt);
    \draw[fill=blue, color=blue] (barycentric cs:x1=1,y1=1,z2=1,y6=1)--(barycentric cs:z1=1,z2=1,z7=1,z5=1);
    \draw[fill=blue, color=blue] (barycentric cs:x2=1,y2=1,z1=1,y1=1)--(barycentric cs:z1=1,z2=1,z7=1,z5=1);
    \draw[fill=blue, color=blue] (barycentric cs:x6=1,y5=1,y6=1,z7=1)--(barycentric cs:z1=1,z2=1,z7=1,z5=1);
    \draw[fill=blue, color=blue] (barycentric cs:z3=1,z4=1,z5=1,z6=1)--(barycentric cs:z1=1,z2=1,z7=1,z5=1);

    \draw[fill=blue, color=blue] (barycentric cs:z3=1,z4=1,z5=1,z6=1)--(barycentric cs:x3=1,y3=1,z4=1,y2=1);
    \draw[fill=blue, color=blue] (barycentric cs:z3=1,z4=1,z5=1,z6=1)--(barycentric cs:x4=1,y4=1,z3=1,y3=1);
    \draw[fill=blue, color=blue] (barycentric cs:z3=1,z4=1,z5=1,z6=1)--(barycentric cs:x5=1,y5=1,z6=1,y4=1);

\draw[fill=blue, color=blue] (4cm,0) circle (2pt);
    \draw[fill=blue, color=blue] (4cm,0)--(barycentric cs:x1=1,y1=1,z2=1,y6=1);
    \draw[fill=blue, color=blue] (4cm,0)--(barycentric cs:x2=1,y2=1,z1=1,y1=1);
    \draw [blue] plot [smooth] coordinates {(4cm,0) (3cm,2cm) (2cm,3cm) (1cm,3cm) (barycentric cs:x3=1,y3=1,z4=1,y2=1)};
    \draw [blue] plot [smooth] coordinates {(4cm,0) (3.5cm,-2.25cm) (2.75cm,-3.25cm) (2cm,-3.66cm) (1cm,-3.72cm) (-0.9cm,-3cm) (barycentric cs:x4=1,y4=1,z3=1,y3=1)};
    \draw [blue] plot [smooth] coordinates {(4cm,0) (3cm,-2cm) (2cm,-3cm) (1cm,-3cm) (barycentric cs:x5=1,y5=1,z6=1,y4=1)};
    \draw[fill=blue, color=blue] (4cm,0)--(barycentric cs:x6=1,y5=1,y6=1,z7=1);
\end{scope}
\end{scope}
\end{tikzpicture}

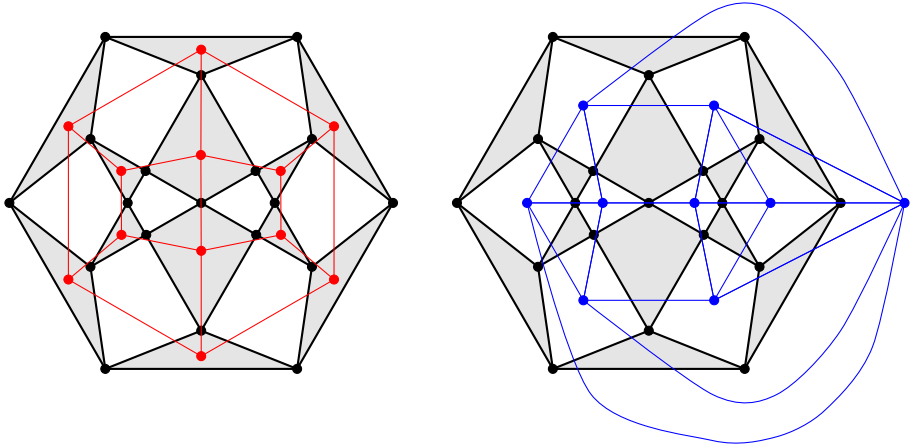
\captionof{figure}{The chess coloring and the related graphs}
\end{center}

Consider a plane graph $\Gamma$. 
The {\it medial graph} of $\Gamma$ is the graph $\mathcal{M}(\Gamma)$ whose vertex set is the edge set of $\Gamma$ and two vertices of $\mathcal{M}(\Gamma)$ are joined by an edge if they have a common vertex and belong to the same face in $\Gamma$. 
The graph $\mathcal{M}(\Gamma)$ is also plane. 
This graph is $4$-regular and its face set is the union of the vertex set and the face set of $\Gamma$. 
Thus, $\mathcal{M}(\Gamma)$ is chess colored. 
Let $b$ be the color used to coloring the faces of $\mathcal{M}(\Gamma)$ corresponding to the vertices of $\Gamma$. 
The remaining faces of $\mathcal{M}(\Gamma)$ (corresponding to the faces of $\Gamma$) are colored in $w$. 
Then 
$$\mathcal{R}_b(\mathcal{M}(\Gamma))=\Gamma\hspace{0.5cm}\text{ and }\hspace{0.5cm}\mathcal{R}_{w}(\mathcal{M}(\Gamma))=\Gamma^*.$$ 
For example, the graph marked in black in Fig. 3 is the medial graph of the graphs marked in red and blue.

A {\it central circuit} is a circuit in the medial graph which is obtained by starting with an edge and continuing at each vertex by the edge opposite to the entering one \cite[p.5]{DDS-book}.
We will consider central circuits as cyclic sequences of vertices and distinguish each central circuit from the reversed. 
If $\Gamma$ and $\Gamma^*$ both are simple, then there is a one-to-one correspondence between zigzags in $\Gamma$ and central circuits in $\mathcal{M}(\Gamma)$: 
a sequence formed by edges of $\Gamma$ is a zigzag if and only if this sequence is a central circuit in $\mathcal{M}(\Gamma)$. 
In Fig. 4. the part of the zigzag is marked by the bold red line and the corresponding part of the central circuit is marked by the bold black line.

\begin{center}
\begin{tikzpicture} 
\begin{scope}[scale=0.8]

\fill [opacity=0.2,gray] (0:3cm)--(30:2cm)--(60:3cm) -- cycle;
\fill [opacity=0.2,gray] (120:3cm)--(90:2cm)--(60:3cm) -- cycle;
\fill [opacity=0.2,gray] (120:3cm)--(150:2cm)--(180:3cm) -- cycle;
\fill [opacity=0.2,gray] (240:3cm)--(210:2cm)--(180:3cm) -- cycle;
\fill [opacity=0.2,gray] (240:3cm)--(270:2cm)--(300:3cm) -- cycle;
\fill [opacity=0.2,gray] (0:3cm)--(330:2cm)--(300:3cm) -- cycle;

\fill [opacity=0.2,gray] (30:2cm)--(0.85cm,0.5cm)--(1.15cm,0cm) -- cycle;
\fill [opacity=0.2,gray] (90:2cm)--(-0.87cm,0.5cm)--(0,0)--(0.85cm,0.5cm) -- cycle;
\fill [opacity=0.2,gray] (150:2cm)--(-0.87cm,0.5cm)--(-1.15cm,0cm) -- cycle;
\fill [opacity=0.2,gray] (210:2cm)--(-1.15cm,0cm)--(-0.86cm,-0.5cm) -- cycle;
\fill [opacity=0.2,gray] (270:2cm)--(-0.86cm,-0.5cm)--(0,0)-- (0.86cm,-0.5cm) -- cycle;
\fill [opacity=0.2,gray] (330:2cm)--(0.86cm,-0.5cm)--(1.15cm,0cm) -- cycle;

\draw[fill=black] (0:3cm) circle (2pt);
\draw[fill=black] (60:3cm) circle (2pt);
\draw[fill=black] (120:3cm) circle (2pt);
\draw[fill=black] (180:3cm) circle (2pt);
\draw[fill=black] (240:3cm) circle (2pt);
\draw[fill=black] (300:3cm) circle (2pt);
    \draw[thick] (0:3cm) \foreach \x in {60, 120,...,359} {
            -- (\x:3cm) 
        } -- cycle (60:3cm);

\draw[fill=black] (30:2cm) circle (2pt);
\draw[fill=black] (90:2cm) circle (2pt);
\draw[fill=black] (150:2cm) circle (2pt);
\draw[fill=black] (210:2cm) circle (2pt);
\draw[fill=black] (270:2cm) circle (2pt);
\draw[fill=black] (330:2cm) circle (2pt);

    \draw (0:3cm)--(30:2cm)--(60:3cm);
    \draw (120:3cm)--(90:2cm)--(60:3cm);
    \draw (120:3cm)--(150:2cm)--(180:3cm);
    \draw (240:3cm)--(210:2cm)--(180:3cm);
    \draw (240:3cm)--(270:2cm)--(300:3cm);
    \draw (0:3cm)--(330:2cm)--(300:3cm);

    \draw (30:2cm)--(270:2cm);
    \draw (30:2cm)--(210:2cm);

    \draw (90:2cm)--(330:2cm);
    \draw (90:2cm)--(210:2cm);

\draw[fill=black] (1.15cm,0cm) circle (2pt);
\node (z2) at (1.15cm,0cm) {};
\draw[fill=black] (0.85cm,0.5cm) circle (2pt);
\node (z1) at (0.85cm,0.5cm) {};

    \draw[thick] (150:2cm)--(270:2cm);
    \draw[thick] (150:2cm)--(330:2cm);

\draw[fill=black] (-1.15cm,0cm) circle (2pt);
\node (z3) at (-1.15cm,0cm) {};
\draw[fill=black] (-0.87cm,0.5cm) circle (2pt);
\node (z4) at (-0.87cm,0.5cm) {};
\draw[fill=black] (0,0) circle (2pt);
\node (z5) at (0,0) {};
\draw[fill=black] (-0.86cm,-0.5cm) circle (2pt);
\node (z6) at (-0.86cm,-0.5cm) {};
\draw[fill=black] (0.86cm,-0.5cm) circle (2pt);
\node (z7) at (0.86cm,-0.5cm) {};

\node (x1) at (0:3cm) {};
\node (x2) at (60:3cm) {};
\node (x3) at (120:3cm) {};
\node (x4) at (180:3cm) {};
\node (x5) at (240:3cm) {};
\node (x6) at (300:3cm) {};

\node (y1) at (30:2cm) {};
\node (y2) at (90:2cm) {};
\node (y3) at (150:2cm) {};
\node (y4) at (210:2cm) {};
\node (y5) at (270:2cm) {};
\node (y6) at (330:2cm) {};
    \draw[thick, line width=1.5pt] (120:3cm)--(150:2cm);
    \draw[thick, line width=1.5pt] (150:2cm)--(-1.15cm,0cm);
    \draw[thick, line width=1.5pt] (-1.15cm,0cm)--(-0.86cm,-0.5cm); 
    \draw[thick, line width=1.5pt] (-0.86cm,-0.5cm)--(270:2cm);
    \draw[thick, line width=1.5pt] (270:2cm)--(300:3cm);
    \draw[thick, line width=1.5pt] (300:3cm)--(0:3cm);
    \draw[thick, line width=1.5pt] (0:3cm)--(30:2cm);
    \draw[thick, line width=1.5pt] (30:2cm)--(0.85cm,0.5cm);


\draw[fill=red, color=red] (barycentric cs:x1=1,x2=1,y1=1) circle (2pt);
\draw[fill=red, color=red] (barycentric cs:x2=1,x3=1,y2=1) circle (2pt);
\draw[fill=red, color=red] (barycentric cs:x3=1,x4=1,y3=1) circle (2pt);
\draw[fill=red, color=red] (barycentric cs:x4=1,x5=1,y4=1) circle (2pt);
\draw[fill=red, color=red] (barycentric cs:x5=1,x6=1,y5=1) circle (2pt);
\draw[fill=red, color=red] (barycentric cs:x6=1,x1=1,y6=1) circle (2pt);
    \draw[color=red] (barycentric cs:x1=1,x2=1,y1=1)--(barycentric cs:x2=1,x3=1,y2=1)--(barycentric cs:x3=1,x4=1,y3=1)--(barycentric cs:x4=1,x5=1,y4=1)--(barycentric cs:x5=1,x6=1,y5=1)--(barycentric cs:x6=1,x1=1,y6=1)--cycle;

\draw[fill=red, color=red] (barycentric cs:z1=1,z2=1,y1=1) circle (2pt);
\draw[fill=red, color=red] (barycentric cs:z1=1,z4=1,y2=1,z5=1) circle (2pt);
\draw[fill=red, color=red] (barycentric cs:z4=1,z3=1,y3=1) circle (2pt);
\draw[fill=red, color=red] (barycentric cs:z3=1,z6=1,y4=1) circle (2pt);
\draw[fill=red, color=red] (barycentric cs:z5=1,z6=1,y5=1,z7=1) circle (2pt);
\draw[fill=red, color=red] (barycentric cs:z2=1,z7=1,y6=1) circle (2pt); 
    \draw[color=red] (barycentric cs:z1=1,z2=1,y1=1)--(barycentric cs:z1=1,z4=1,y2=1,z5=1)--(barycentric cs:z4=1,z3=1,y3=1)--(barycentric cs:z3=1,z6=1,y4=1)--(barycentric cs:z5=1,z6=1,y5=1,z7=1)--(barycentric cs:z2=1,z7=1,y6=1)--cycle;
    \draw[color=red] (barycentric cs:z1=1,z4=1,y2=1,z5=1)--(barycentric cs:z5=1,z6=1,y5=1,z7=1);

    \draw[color=red] (barycentric cs:x1=1,x2=1,y1=1)--(barycentric cs:y1=1,z1=1,z2=1);
    \draw[color=red] (barycentric cs:x2=1,x3=1,y2=1)--(barycentric cs:y2=1,z1=1,z4=1,z5=1);
    \draw[color=red] (barycentric cs:x3=1,x4=1,y3=1)--(barycentric cs:y3=1,z3=1,z4=1);
    \draw[color=red] (barycentric cs:x4=1,x5=1,y4=1)--(barycentric cs:y4=1,z3=1,z6=1);
    \draw[color=red] (barycentric cs:x5=1,x6=1,y5=1)--(barycentric cs:y5=1,z6=1,z7=1,z5=1);
    \draw[color=red] (barycentric cs:x6=1,x1=1,y6=1)--(barycentric cs:y6=1,z2=1,z7=1);
\draw[color=red, line width=1.5pt] (barycentric cs:x2=1,x3=1,y2=1)--(barycentric cs:x3=1,x4=1,y3=1)--(barycentric cs:z4=1,z3=1,y3=1)--(barycentric cs:z3=1,z6=1,y4=1)--(barycentric cs:z5=1,z6=1,y5=1,z7=1)--(barycentric cs:x5=1,x6=1,y5=1)--(barycentric cs:x6=1,x1=1,y6=1)--(barycentric cs:x1=1,x2=1,y1=1)--(barycentric cs:z1=1,z2=1,y1=1)--(barycentric cs:z1=1,z4=1,y2=1,z5=1);

\end{scope}
\end{tikzpicture}

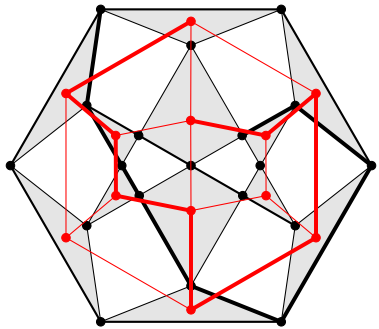
\captionof{figure}{A zigzag and the corresponding central circuit}
\end{center}

\subsection{Main construction}
Let $\sigma$ be a permutation on $[k]_{\pm}$ satisfying (M1) and (M2). 
We construct a $4$-regular plane graph $G$ which is the medial graph of a plane graph, where $\sigma$ realizes as the $z$-monodromy of a face. 

Consider a circle $C$ embedded in the plane. 
We take mutually distinct points $p_1,\dots,p_k$ from $C$ such that these points occur along $C$ in the clockwise order. 
Next, denote by $C'$ a circle inside the part of the plane bounded by $C$ and take mutually distinct points $a_{12}, a_{23},\dots, a_{(k-1)k}, a_{k1}$ occurring on $C'$ in the clockwise order. 
Similarly, let $C''$ be a circle inside the part of the plane bounded by $C'$ and let $r_1,l_k,r_2,l_1,\dots,r_k,l_{k-1}$ be mutually distinct points that occur along $C''$ in the clockwise order. 
For every $a_{ij}$ we take two 
segments that intersect precisely in $a_{ij}$ and join $p_i$ with $l_i$ and $p_j$ with $r_j$, respectively. 
Note that all such 
segments intersect each of $C, C', C''$ in precisely one point and the interiors of any two of these 
segments are disjoint if they contain distinct points $a_{ij}$. 
Denote by $S_i$ the union of the segment joining $r_i$ with $p_i$, the arc of $C$ between $p_i$ and $p_{i+1}$ and the segment joining $p_{i+1}$ with $l_{i+1}$ (indices are taken modulo $k$), see Fig. 5 for the case $k=6$. 

\begin{center}
\begin{tikzpicture}[scale=0.9]
\draw (0,0) circle (3cm);

    \coordinate (P1) at (60:3cm);
    \coordinate (P2) at (0:3cm);
    \coordinate (P3) at (300:3cm);
    \coordinate (P4) at (240:3cm);
    \coordinate (P5) at (180:3cm);
    \coordinate (P6) at (120:3cm);

    \coordinate (A12) at (30:2.25cm);
    \coordinate (A23) at (330:2.25cm);
    \coordinate (A34) at (270:2.25cm);
    \coordinate (A45) at (210:2.25cm);
    \coordinate (A56) at (150:2.25cm);
    \coordinate (A61) at (90:2.25cm);

\draw[fill=black] (P1) circle (1pt);
\draw[fill=black] (P2) circle (1pt);
\draw[fill=black] (P3) circle (1pt);
\draw[fill=black] (P4) circle (1pt);
\draw[fill=black] (P5) circle (1pt);
\draw[fill=black] (P6) circle (1pt);

\draw[fill=black] (A12) circle (1pt);
\draw[fill=black] (A23) circle (1pt);
\draw[fill=black] (A34) circle (1pt);
\draw[fill=black] (A45) circle (1pt);
\draw[fill=black] (A56) circle (1pt);
\draw[fill=black] (A61) circle (1pt);

    \coordinate (L1) at ($(P1)!1.5!(A12)$);
    \coordinate (L2) at ($(P2)!1.5!(A23)$);
    \coordinate (L3) at ($(P3)!1.5!(A34)$);
    \coordinate (L4) at ($(P4)!1.5!(A45)$);
    \coordinate (L5) at ($(P5)!1.5!(A56)$);
    \coordinate (L6) at ($(P6)!1.5!(A61)$);

    \coordinate (R1) at ($(P1)!1.5!(A61)$);
    \coordinate (R2) at ($(P2)!1.5!(A12)$);
    \coordinate (R3) at ($(P3)!1.5!(A23)$);
    \coordinate (R4) at ($(P4)!1.5!(A34)$);
    \coordinate (R5) at ($(P5)!1.5!(A45)$);
    \coordinate (R6) at ($(P6)!1.5!(A56)$);

    \draw (P1) -- (L1);
    \draw (P1) -- (R1);
    \draw (P2) -- (L2);
    \draw (P2) -- (R2);
    \draw (P3) -- (L3);
    \draw (P3) -- (R3);
    \draw (P4) -- (L4);
    \draw (P4) -- (R4);
    \draw (P5) -- (L5);
    \draw (P5) -- (R5);
    \draw (P6) -- (L6);
    \draw (P6) -- (R6);

\draw[fill=black] (L1) circle (1pt);
\draw[fill=black] (L2) circle (1pt);
\draw[fill=black] (L3) circle (1pt);
\draw[fill=black] (L4) circle (1pt);
\draw[fill=black] (L5) circle (1pt);
\draw[fill=black] (L6) circle (1pt);

\draw[fill=black] (R1) circle (1pt);
\draw[fill=black] (R2) circle (1pt);
\draw[fill=black] (R3) circle (1pt);
\draw[fill=black] (R4) circle (1pt);
\draw[fill=black] (R5) circle (1pt);
\draw[fill=black] (R6) circle (1pt);

\node at (60:3.3cm) {$p_1$};
\node at (0:3.3cm) {$p_2$};
\node at (300:3.3cm) {$p_3$};
\node at (240:3.3cm) {$p_4$};
\node at (180:3.3cm) {$p_5$};
\node at (120:3.3cm) {$p_6$};

\node at (28:2.6cm) {$a_{12}$};
\node at (90:2.5cm) {$a_{61}$};
\node at (152:2.62cm) {$a_{56}$};
\node at (208:2.6cm) {$a_{45}$};
\node at (270:2.54cm) {$a_{34}$};
\node at (332:2.62cm) {$a_{23}$};

\node at ($(0,0)!0.88!(L1)$) {$l_1$};
\node at ($(0,0)!0.88!(L2)$) {$l_2$};
\node at ($(0,0)!0.88!(L3)$) {$l_3$};
\node at ($(0,0)!0.88!(L4)$) {$l_4$};
\node at ($(0,0)!0.88!(L5)$) {$l_5$};
\node at ($(0,0)!0.88!(L6)$) {$l_6$};

\node at ($(0,0)!0.88!(R1)$) {$r_1$};
\node at ($(0,0)!0.88!(R2)$) {$r_2$};
\node at ($(0,0)!0.88!(R3)$) {$r_3$};
\node at ($(0,0)!0.88!(R4)$) {$r_4$};
\node at ($(0,0)!0.88!(R5)$) {$r_5$};
\node at ($(0,0)!0.88!(R6)$) {$r_6$};

    \coordinate (S1) at (30:3.35cm);
    \coordinate (S2) at (330:3.35cm);
    \coordinate (S3) at (270:3.35cm);
    \coordinate (S4) at (210:3.35cm);
    \coordinate (S5) at (150:3.35cm);
    \coordinate (S6) at (90:3.35cm);

\node at (S1) {$S_1$};
\node at (S2) {$S_2$};
\node at (S3) {$S_3$};
\node at (S4) {$S_4$};
\node at (S5) {$S_5$};
\node at (S6) {$S_6$};
\end{tikzpicture}

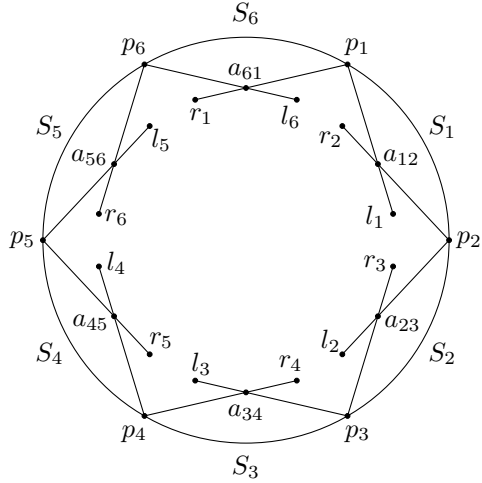
\captionof{figure}{The beginning of construction for $k=6$}
\end{center}

We will work with the relation $\sim$ on the set 
$$\mathcal{O}=\bigcup\limits_{i=1}^{k}\{l_i,r_i\}$$
such that for any $i,j\in[k]_\pm$ satisfying $\sigma(i)=j$ one of the following possibilities is realized: 
\begin{enumerate}
\item[(1)] $l_i\sim r_j$ if $i,j\in[k]_+$, 
\item[(2)] $r_{-i}\sim l_{-j}$ if $i,j\in[k]_-$, 
\item[(3)] $l_i\sim l_{-j}$ if $i\in[k]_+$ and $j\in[k]_-$, 
\item[(4)] $r_{-i}\sim r_j$ if $i\in[k]_-$ and $j\in[k]_+$. 
\end{enumerate}
The relation $\sim$ is irreflexive and symmetric. 
Indeed, if $l_i\sim l_i$ (the case (3)) or $r_i\sim r_i$ (the case (4)), then we get $\sigma(i)=-i$ which contradicts (M2). 
Thus, $\sim$ is irreflexive. 
Now, we show that if $l_i\sim l_j$, then $l_j\sim l_i$ (the remaining three cases are similar). 
If $l_i\sim l_j$ with $i,j\in[k]_+$ (the case (3)), then $\sigma(i)=-j$ and, by (M1), $\sigma(j)=-i$ and $j\in[k]_+, -i\in[k]_-$; i.e. $l_j\sim l_i$.
Note that for each $x\in\mathcal{O}$ there is a unique $x'\in\mathcal{O}$ such that $x\sim x'$. 

If two points from $\mathcal{O}$ are in the relation $\sim$, then we join them by a curve homeomorphic to the segment $[0,1]$ inside the part of the plane bounded by $C''$. 
The following conditions must be satisfied:

$\bullet$ the curves have no more than finitely many intersections and self-intersections, 

$\bullet$ for every such intersection point either there are precisely two distinct curves passing once through this point or there is a single curve passing twice through it, 

$\bullet$ all intersections are transversal.


\noindent Denote the family of the curves described above by $\mathcal{L}$. 

Note that every $x\in\mathcal{O}$ is a common point for a unique $S_i$ and a unique $L\in\mathcal{L}$. 
Thus, for every curve from $\{S_1,\dots,S_k\}$ we identify each of its endpoints with the corresponding endpoint of a curve from $\mathcal{L}$ 
and obtain a collection of closed curves which will be denoted by $\mathcal{C}$; curves from this collection need not be disjoint. 
Let $V$ be the set of all intersection and self-intersection points of curves from $\mathcal{C}$. 
In particular, all $p_i$ and all $a_{ij}$ belong to $V$. 

\begin{exmp}\label{ex3}{\rm
Let $k=6$ and 
$$\sigma=(1,-6,-4,2)(3,-5)(5,-3)(-2,4,6,-1).$$
The relation $\sim$ on the set $\mathcal{O}=\{l_1.\dots,l_6,r_1,\dots,r_6\}$ is as follows
$$l_1\sim l_6,\hspace{0.2cm}r_6\sim l_4,\hspace{0.2cm}r_4\sim r_2,\hspace{0.2cm}l_2\sim r_1,\hspace{0.2cm}l_3\sim l_5,\hspace{0.2cm}r_5\sim r_3.$$
One of suitable connections between points from $\mathcal{O}$ is presented in Fig. 6. 
\begin{center}
\begin{tikzpicture}[scale=0.9]

    \coordinate (P1) at (60:3cm);
    \coordinate (P2) at (0:3cm);
    \coordinate (P3) at (300:3cm);
    \coordinate (P4) at (240:3cm);
    \coordinate (P5) at (180:3cm);
    \coordinate (P6) at (120:3cm);

    \coordinate (A12) at (30:2.25cm);
    \coordinate (A23) at (330:2.25cm);
    \coordinate (A34) at (270:2.25cm);
    \coordinate (A45) at (210:2.25cm);
    \coordinate (A56) at (150:2.25cm);
    \coordinate (A61) at (90:2.25cm);

    \coordinate (L1) at ($(P1)!1.5!(A12)$);
    \coordinate (L2) at ($(P2)!1.5!(A23)$);
    \coordinate (L3) at ($(P3)!1.5!(A34)$);
    \coordinate (L4) at ($(P4)!1.5!(A45)$);
    \coordinate (L5) at ($(P5)!1.5!(A56)$);
    \coordinate (L6) at ($(P6)!1.5!(A61)$);

    \coordinate (R1) at ($(P1)!1.5!(A61)$);
    \coordinate (R2) at ($(P2)!1.5!(A12)$);
    \coordinate (R3) at ($(P3)!1.5!(A23)$);
    \coordinate (R4) at ($(P4)!1.5!(A34)$);
    \coordinate (R5) at ($(P5)!1.5!(A45)$);
    \coordinate (R6) at ($(P6)!1.5!(A56)$);

    \draw[color=blue] (P1) -- (L1);
    \draw[color=green] (P1) -- (R1);
\draw [green,domain=0:60] plot ({3cm*cos(\x)}, {3cm*sin(\x)});
    \draw[color=green] (P2) -- (L2);
    \draw[color=red] (P2) -- (R2);
\draw [red,domain=300:360] plot ({3cm*cos(\x)}, {3cm*sin(\x)});
    \draw[color=red] (P3) -- (L3);
    \draw[color=blue] (P3) -- (R3);
\draw [blue,domain=240:300] plot ({3cm*cos(\x)}, {3cm*sin(\x)});
    \draw[color=blue] (P4) -- (L4);
    \draw[color=red] (P4) -- (R4);
    \draw[color=red] (P5) -- (L5);
\draw [red,domain=180:240] plot ({3cm*cos(\x)}, {3cm*sin(\x)});
    \draw[color=blue] (P5) -- (R5);
    \draw[color=blue] (P6) -- (L6);
\draw [blue,domain=120:180] plot ({3cm*cos(\x)}, {3cm*sin(\x)});
    \draw[color=blue] (P6) -- (R6);
\draw [blue,domain=60:120] plot ({3cm*cos(\x)}, {3cm*sin(\x)});

\node at (60:3.3cm) {$p_1$};
\node at (0:3.3cm) {$p_2$};
\node at (300:3.3cm) {$p_3$};
\node at (240:3.3cm) {$p_4$};
\node at (180:3.3cm) {$p_5$};
\node at (120:3.3cm) {$p_6$};

\node at (28:2.6cm) {$a_{12}$};
\node at (90:2.5cm) {$a_{61}$};
\node at (152:2.62cm) {$a_{56}$};
\node at (208:2.6cm) {$a_{45}$};
\node at (270:2.54cm) {$a_{34}$};
\node at (332:2.62cm) {$a_{23}$};

\node at ($($(0,0)!0.88!(L1)$)+(0.06cm,0.2cm)$) {$l_1$};
\node at ($($(0,0)!0.88!(L2)$)+(0.15cm,0.1cm)$) {$l_2$};
\node at ($(0,0)!0.88!(L3)$) {$l_3$};
\node at ($(0,0)!0.88!(L4)$) {$l_4$};
\node at ($(0,0)!0.88!(L5)$) {$l_5$};
\node at ($(0,0)!0.88!(L6)$) {$l_6$};

\node at ($(0,0)!0.88!(R1)$) {$r_1$};
\node at ($($(0,0)!0.88!(R2)$)+(0.18cm,-0.08cm)$) {$r_2$};
\node at ($($(0,0)!0.88!(R3)$)+(0,-0.1cm)$) {$r_3$};
\node at ($(0,0)!0.88!(R4)$) {$r_4$};
\node at ($(0,0)!0.88!(R5)$) {$r_5$};
\node at ($(0,0)!0.88!(R6)$) {$r_6$};

    \draw[color=blue] (L4) -- (R6);
\draw [red] plot [smooth] coordinates {(R2) ($(R2)+(-0.15cm,-0.1cm)$) ($(0,0)!0.6!($0.5*(R2)+0.5*(R4)$)$) ($(R4)+(0.15cm,0.3cm)$) (R4)};
\draw [green] plot [smooth] coordinates {(L2) ($(L2)+(-0.36cm,0.15cm)$) ($($(0,0)!0.1!($0.5*(L2)+0.5*(R1)$)$)+(-0.1cm,-0.1cm)$) ($(R1)+(-0.15cm,-0.3cm)$) (R1)};
\draw [blue] plot [smooth] coordinates {(L1) ($(L1)+(-0.1cm,-0.1cm)$) ($(0,0)!0.6!($0.5*(L1)+0.5*(L6)$)$) ($(L6)+(0.15cm,-0.3cm)$) (L6)};
\draw [red] plot [smooth] coordinates {(L5) ($(L5)+(0.3cm,-0.1cm)$) ($(0,0)!0.6!($0.5*(L5)+0.5*(L3)$)$) ($(L3)+(-0.2cm,0.3cm)$) (L3)};
\draw [blue] plot [smooth] coordinates {(R5) ($(R5)+(0.4cm,0cm)$) ($(0,0)!0.6!($0.5*(R5)+0.5*(R3)$)$) ($(R3)+(-0.2cm,0.2cm)$) (R3)};
\draw[fill=black] (P1) circle (1.25pt);
\draw[fill=black] (P2) circle (1.25pt);
\draw[fill=black] (P3) circle (1.25pt);
\draw[fill=black] (P4) circle (1.25pt);
\draw[fill=black] (P5) circle (1.25pt);
\draw[fill=black] (P6) circle (1.25pt);

\draw[fill=black] (A12) circle (1.25pt);
\draw[fill=black] (A23) circle (1.25pt);
\draw[fill=black] (A34) circle (1.25pt);
\draw[fill=black] (A45) circle (1.25pt);
\draw[fill=black] (A56) circle (1.25pt);
\draw[fill=black] (A61) circle (1.25pt);

\draw[fill=white] (L1) circle (1.25pt);
\draw[fill=white] (L2) circle (1.25pt);
\draw[fill=white] (L3) circle (1.25pt);
\draw[fill=white] (L4) circle (1.25pt);
\draw[fill=white] (L5) circle (1.25pt);
\draw[fill=white] (L6) circle (1.25pt);

\draw[fill=white] (R1) circle (1.25pt);
\draw[fill=white] (R2) circle (1.25pt);
\draw[fill=white] (R3) circle (1.25pt);
\draw[fill=white] (R4) circle (1.25pt);
\draw[fill=white] (R5) circle (1.25pt);
\draw[fill=white] (R6) circle (1.25pt);


\draw[fill=black] (-0.925cm,-1.61cm) circle (1.25pt);
\draw[fill=black] (0.795cm,-1.235cm) circle (1.25pt);
\draw[fill=black] (0.660cm,-0.47cm) circle (1.25pt);
\draw[fill=black] (0.286cm,-0.601cm) circle (1.25pt);
\draw[fill=black] (0.905cm,0.72cm) circle (1.25pt);
\end{tikzpicture}

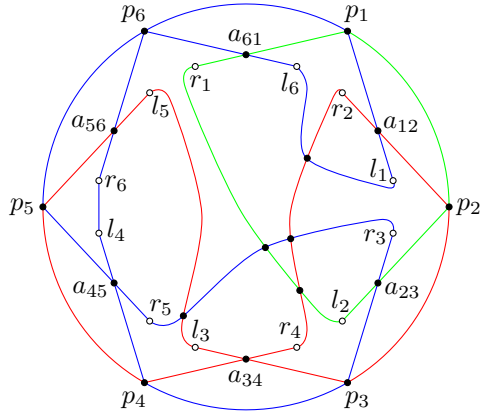
\captionof{figure}{A suitable connection between points of $\mathcal{O}$}
\end{center}
In this case, $\mathcal{C}$ consists of precisely $3$ closed curves with $17$ intersection points, i.e. $|V|=17$.
}\end{exmp}

Let $G=G(\mathcal{C})$ be the graph whose vertex set is $V$ and two vertices are joined by an edge if they are two consecutive points on one of curves from $\mathcal{C}$. 
In fact, we consider $G$ as a graph whose vertices are points on the plane and edges are parts of curves from $\mathcal{C}$ joining these points. 
It is easy to see that $G$ is a $4$-regular and the curves from $\mathcal{C}$ correspond to pairs of mutually reversed central circuits from $G$. 
We make the chess coloring of $G$ and split the set of its faces into the two sets corresponding to the colors. 
Let $b$ be the color of faces whose boundaries are cycles with vertices $p_i, p_j, a_{ij}$; as above, $w$ is the other color. 

As in Subsection 4.1, we obtain the dual graphs $\mathcal{R}_b(G)$ and $\mathcal{R}_{w}(G)$. 
These graphs are not necessarily simple, i.e. they may contain the following fragments: 

(A) loops, 

(B) multiple edges, 

(C) edges that belong to the boundary of one face only (in particular, edges with vertices of degree $1$), 

(D) pairs of faces whose intersection of boundaries contains more than one edge. 

\noindent If one of the cases (A) or (B) realizes in one of the graphs $\mathcal{R}_b(G), \mathcal{R}_{w}(G)$, then the case (C) or (D), respectively, realizes in the dual graph. 
\begin{exmp}\label{ex4}{\rm
Let $\mathcal{C}$ be as in Example \ref{ex3}. 
The graphs $\mathcal{R}_b(G)$ and $\mathcal{R}_{w}(G)$ are presented in Fig. 7a and Fig. 7b, respectively. 
\begin{center}
\begin{tikzpicture}[scale=0.8]
\begin{scope}
    \coordinate (P1) at (60:3cm);
    \coordinate (P2) at (0:3cm);
    \coordinate (P3) at (300:3cm);
    \coordinate (P4) at (240:3cm);
    \coordinate (P5) at (180:3cm);
    \coordinate (P6) at (120:3cm);

    \coordinate (A12) at (30:2.25cm);
    \coordinate (A23) at (330:2.25cm);
    \coordinate (A34) at (270:2.25cm);
    \coordinate (A45) at (210:2.25cm);
    \coordinate (A56) at (150:2.25cm);
    \coordinate (A61) at (90:2.25cm);

    \coordinate (L1) at ($(P1)!1.5!(A12)$);
    \coordinate (L2) at ($(P2)!1.5!(A23)$);
    \coordinate (L3) at ($(P3)!1.5!(A34)$);
    \coordinate (L4) at ($(P4)!1.5!(A45)$);
    \coordinate (L5) at ($(P5)!1.5!(A56)$);
    \coordinate (L6) at ($(P6)!1.5!(A61)$);

    \coordinate (R1) at ($(P1)!1.5!(A61)$);
    \coordinate (R2) at ($(P2)!1.5!(A12)$);
    \coordinate (R3) at ($(P3)!1.5!(A23)$);
    \coordinate (R4) at ($(P4)!1.5!(A34)$);
    \coordinate (R5) at ($(P5)!1.5!(A45)$);
    \coordinate (R6) at ($(P6)!1.5!(A56)$);

    \draw[color=black!40] (P1) -- (L1);
    \draw[color=black!40] (P1) -- (R1);
\draw [black!40,domain=0:60] plot ({3cm*cos(\x)}, {3cm*sin(\x)});
    \draw[color=black!40] (P2) -- (L2);
    \draw[color=black!40] (P2) -- (R2);
\draw [black!40,domain=300:360] plot ({3cm*cos(\x)}, {3cm*sin(\x)});
    \draw[color=black!40] (P3) -- (L3);
    \draw[color=black!40] (P3) -- (R3);
\draw [black!40,domain=240:300] plot ({3cm*cos(\x)}, {3cm*sin(\x)});
    \draw[color=black!40] (P4) -- (L4);
    \draw[color=black!40] (P4) -- (R4);
    \draw[color=black!40] (P5) -- (L5);
\draw [black!40,domain=180:240] plot ({3cm*cos(\x)}, {3cm*sin(\x)});
    \draw[color=black!40] (P5) -- (R5);
    \draw[color=black!40] (P6) -- (L6);
\draw [black!40,domain=120:180] plot ({3cm*cos(\x)}, {3cm*sin(\x)});
    \draw[color=black!40] (P6) -- (R6);
\draw [black!40,domain=60:120] plot ({3cm*cos(\x)}, {3cm*sin(\x)});

\node[color=black!40] at (60:3.3cm) {$p_1$};
\node[color=black!40] at (0:3.3cm) {$p_2$};
\node[color=black!40] at (300:3.3cm) {$p_3$};
\node[color=black!40] at (240:3.3cm) {$p_4$};
\node[color=black!40] at (180:3.3cm) {$p_5$};
\node[color=black!40] at (120:3.3cm) {$p_6$};

\node[color=black!40] at ($(28:2.6cm)+(-0.78cm,0cm)$) {$a_{12}$};
\node[color=black!40] at ($(90:2.5cm)+(-0.175cm,-0.5cm)$) {$a_{61}$};
\node[color=black!40] at (152:1.90cm) {$a_{56}$};
\node[color=black!40] at ($(208:1.88cm)+(0.03cm,-0.15cm)$) {$a_{45}$};
\node[color=black!40] at ($(270:2.54cm)+(-0.3cm,0.55cm)$) {$a_{34}$};
\node[color=black!40] at ($(332:2.62cm)+(-0.6cm,0.36cm)$) {$a_{23}$};

\node[color=black!40] at ($($(0,0)!0.88!(L1)$)+(0.5cm,0.0cm)$) {$l_1$};
\node[color=black!40] at ($($(0,0)!0.88!(L2)$)+(0.18cm,-0.5cm)$) {$l_2$};
\node[color=black!40] at ($($(0,0)!0.88!(L3)$)+(-0.2cm,-0.4cm)$) {$l_3$};
\node[color=black!40] at ($($(0,0)!0.88!(L4)$)+(-0.45cm,0)$) {$l_4$};
\node[color=black!40] at ($($(0,0)!0.88!(L5)$)+(-0.17cm,0.47cm)$) {$l_5$};
\node[color=black!40] at ($($(0,0)!0.88!(L6)$)+(0.32cm,0.35cm)$) {$l_6$};

\node[color=black!40] at ($($(0,0)!0.88!(R1)$)+(-0.2cm,0.44cm)$) {$r_1$};
\node[color=black!40] at ($($(0,0)!0.88!(R2)$)+(0.18cm,0.4cm)$) {$r_2$};
\node[color=black!40] at ($($(0,0)!0.88!(R3)$)+(0.5cm,0cm)$) {$r_3$};
\node[color=black!40] at ($($(0,0)!0.88!(R4)$)+(0.23cm,-0.48cm)$) {$r_4$};
\node[color=black!40] at ($($(0,0)!0.88!(R5)$)+(-0.25cm,-0.4cm)$) {$r_5$};
\node[color=black!40] at ($($(0,0)!0.88!(R6)$)+(-0.47cm,-0.02cm)$) {$r_6$};

    \draw[color=black!40] (L4) -- (R6);
\draw [black!40] plot [smooth] coordinates {(R2) ($(R2)+(-0.15cm,-0.1cm)$) ($(0,0)!0.6!($0.5*(R2)+0.5*(R4)$)$) ($(R4)+(0.15cm,0.3cm)$) (R4)};
\draw [black!40] plot [smooth] coordinates {(L2) ($(L2)+(-0.36cm,0.15cm)$) ($($(0,0)!0.1!($0.5*(L2)+0.5*(R1)$)$)+(-0.1cm,-0.1cm)$) ($(R1)+(-0.15cm,-0.3cm)$) (R1)};
\draw [black!40] plot [smooth] coordinates {(L1) ($(L1)+(-0.1cm,-0.1cm)$) ($(0,0)!0.6!($0.5*(L1)+0.5*(L6)$)$) ($(L6)+(0.15cm,-0.3cm)$) (L6)};
\draw [black!40] plot [smooth] coordinates {(L5) ($(L5)+(0.3cm,-0.1cm)$) ($(0,0)!0.6!($0.5*(L5)+0.5*(L3)$)$) ($(L3)+(-0.2cm,0.3cm)$) (L3)};
\draw [black!40] plot [smooth] coordinates {(R5) ($(R5)+(0.4cm,0cm)$) ($(0,0)!0.6!($0.5*(R5)+0.5*(R3)$)$) ($(R3)+(-0.2cm,0.2cm)$) (R3)};
\draw[color=black!40, fill=black!40] (P1) circle (1.25pt);
\draw[color=black!40, fill=black!40] (P2) circle (1.25pt);
\draw[color=black!40, fill=black!40] (P3) circle (1.25pt);
\draw[color=black!40, fill=black!40] (P4) circle (1.25pt);
\draw[color=black!40, fill=black!40] (P5) circle (1.25pt);
\draw[color=black!40, fill=black!40] (P6) circle (1.25pt);

\draw[color=black!40, fill=black!40] (A12) circle (1.25pt);
\draw[color=black!40, fill=black!40] (A23) circle (1.25pt);
\draw[color=black!40, fill=black!40] (A34) circle (1.25pt);
\draw[color=black!40, fill=black!40] (A45) circle (1.25pt);
\draw[color=black!40, fill=black!40] (A56) circle (1.25pt);
\draw[color=black!40, fill=black!40] (A61) circle (1.25pt);

\draw[color=black!40, fill=white] (L1) circle (1.25pt);
\draw[color=black!40, fill=white] (L2) circle (1.25pt);
\draw[color=black!40, fill=white] (L3) circle (1.25pt);
\draw[color=black!40, fill=white] (L4) circle (1.25pt);
\draw[color=black!40, fill=white] (L5) circle (1.25pt);
\draw[color=black!40, fill=white] (L6) circle (1.25pt);

\draw[color=black!40, fill=white] (R1) circle (1.25pt);
\draw[color=black!40, fill=white] (R2) circle (1.25pt);
\draw[color=black!40, fill=white] (R3) circle (1.25pt);
\draw[color=black!40, fill=white] (R4) circle (1.25pt);
\draw[color=black!40, fill=white] (R5) circle (1.25pt);
\draw[color=black!40, fill=white] (R6) circle (1.25pt);


\draw[color=black!40, fill=black!40] (-0.925cm,-1.61cm) circle (1.25pt); 
\draw[color=black!40, fill=black!40] (0.795cm,-1.235cm) circle (1.25pt); 
\draw[color=black!40, fill=black!40] (0.660cm,-0.47cm) circle (1.25pt); 
\draw[color=black!40, fill=black!40] (0.286cm,-0.601cm) circle (1.25pt); 
\draw[color=black!40, fill=black!40] (0.905cm,0.72cm) circle (1.25pt); 

    \coordinate (Q1) at (-0.925cm,-1.61cm);
    \coordinate (Q2) at (0.795cm,-1.235cm);
    \coordinate (Q3) at (0.660cm,-0.47cm);
    \coordinate (Q4) at (0.286cm,-0.601cm);
    \coordinate (Q5) at (0.905cm,0.72cm);

    \coordinate (v1) at (barycentric cs:P1=1,P2=1,A12=1);
    \coordinate (v2) at (barycentric cs:P2=1,P3=1,A23=1);
    \coordinate (v3) at (barycentric cs:P3=1,P4=1,A34=1);
    \coordinate (v4) at (barycentric cs:P4=1,P5=1,A45=1);
    \coordinate (v5) at (barycentric cs:P5=1,P6=1,A56=1);
    \coordinate (v6) at (barycentric cs:P6=1,P1=1,A61=1);
    \coordinate (v7) at (barycentric cs:R1=1,L6=1,A61=1,Q3=1,Q4=1,Q5=1);
    \coordinate (v8) at (barycentric cs:R2=1,L1=1,A12=1,Q5=1);
    \coordinate (v9) at (barycentric cs:L5=1,A56=1,R6=1,L4=1,A45=1,R5=1,Q1=1);
    \coordinate (v10) at (barycentric cs:L3=1,A34=1,R4=1,Q1=1,Q2=1,Q4=1);
    \coordinate (v11) at (barycentric cs:L2=1,A23=1,R3=1,Q2=1,Q3=1);

\draw[fill=red, color=red] (v1) circle (1.75pt);
\draw[fill=red, color=red] (v2) circle (1.75pt);
\draw[fill=red, color=red] (v3) circle (1.75pt);
\draw[fill=red, color=red] (v4) circle (1.75pt);
\draw[fill=red, color=red] (v5) circle (1.75pt);
\draw[fill=red, color=red] (v6) circle (1.75pt);
\draw[fill=red, color=red] (v7) circle (1.75pt);
\draw[fill=red, color=red] (v8) circle (1.75pt);
\draw[fill=red, color=red] (v9) circle (1.75pt);
\draw[fill=red, color=red] (v10) circle (1.75pt);
\draw[fill=red, color=red] (v11) circle (1.75pt);

    \draw[line width=1.1,  color=red] (v1) -- (v2);
    \draw[line width=1.1,  color=red] (v2) -- (v3);
    \draw[line width=1.1,  color=red] (v3) -- (v4);
    \draw[line width=1.1,  color=red] (v4) -- (v5);
    \draw[line width=1.1,  color=red] (v5) -- (v6);
    \draw[line width=1.1,  color=red] (v6) -- (v1);
    \draw[line width=1.1,  color=red] (v1) -- (v8);
    \draw[line width=1.1,  color=red] (v8) -- (v7);
    \draw[line width=1.1,  color=red] (v6) -- (v7);
    \draw[line width=1.1,  color=red] (v10) -- (v7);
    \draw[line width=1.1,  color=red] (v11) -- (v7);
    \draw[line width=1.1,  color=red] (v11) -- (v10);
    \draw[line width=1.1,  color=red] (v11) -- (v2);
    \draw[line width=1.1,  color=red] (v10) -- (v3);
    \draw[line width=1.1,  color=red] (v9) -- (v5);
    \draw[line width=1.1,  color=red] (v9) -- (v4);
    \draw[line width=1.1,  color=red] (v9) -- (v10);
\end{scope}
\begin{scope}[xshift=8cm]
    \coordinate (P1) at (60:3cm);
    \coordinate (P2) at (0:3cm);
    \coordinate (P3) at (300:3cm);
    \coordinate (P4) at (240:3cm);
    \coordinate (P5) at (180:3cm);
    \coordinate (P6) at (120:3cm);

    \coordinate (A12) at (30:2.25cm);
    \coordinate (A23) at (330:2.25cm);
    \coordinate (A34) at (270:2.25cm);
    \coordinate (A45) at (210:2.25cm);
    \coordinate (A56) at (150:2.25cm);
    \coordinate (A61) at (90:2.25cm);

    \coordinate (L1) at ($(P1)!1.5!(A12)$);
    \coordinate (L2) at ($(P2)!1.5!(A23)$);
    \coordinate (L3) at ($(P3)!1.5!(A34)$);
    \coordinate (L4) at ($(P4)!1.5!(A45)$);
    \coordinate (L5) at ($(P5)!1.5!(A56)$);
    \coordinate (L6) at ($(P6)!1.5!(A61)$);

    \coordinate (R1) at ($(P1)!1.5!(A61)$);
    \coordinate (R2) at ($(P2)!1.5!(A12)$);
    \coordinate (R3) at ($(P3)!1.5!(A23)$);
    \coordinate (R4) at ($(P4)!1.5!(A34)$);
    \coordinate (R5) at ($(P5)!1.5!(A45)$);
    \coordinate (R6) at ($(P6)!1.5!(A56)$);

    \draw[color=black!40] (P1) -- (L1);
    \draw[color=black!40] (P1) -- (R1);
\draw [black!40,domain=0:60] plot ({3cm*cos(\x)}, {3cm*sin(\x)});
    \draw[color=black!40] (P2) -- (L2);
    \draw[color=black!40] (P2) -- (R2);
\draw [black!40,domain=300:360] plot ({3cm*cos(\x)}, {3cm*sin(\x)});
    \draw[color=black!40] (P3) -- (L3);
    \draw[color=black!40] (P3) -- (R3);
\draw [black!40,domain=240:300] plot ({3cm*cos(\x)}, {3cm*sin(\x)});
    \draw[color=black!40] (P4) -- (L4);
    \draw[color=black!40] (P4) -- (R4);
    \draw[color=black!40] (P5) -- (L5);
\draw [black!40,domain=180:240] plot ({3cm*cos(\x)}, {3cm*sin(\x)});
    \draw[color=black!40] (P5) -- (R5);
    \draw[color=black!40] (P6) -- (L6);
\draw [black!40,domain=120:180] plot ({3cm*cos(\x)}, {3cm*sin(\x)});
    \draw[color=black!40] (P6) -- (R6);
\draw [black!40,domain=60:120] plot ({3cm*cos(\x)}, {3cm*sin(\x)});

\node[color=black!40] at (69:2.85cm) {$p_1$};
\node[color=black!40] at ($(0:3.3cm)+(-0.03cm,0.2cm)$) {$p_2$};
\node[color=black!40] at ($(300:3.3cm)+(-0.3cm,0.5cm)$) {$p_3$};
\node[color=black!40] at ($(240:3.3cm)+(0.35cm,0.5cm)$) {$p_4$};
\node[color=black!40] at (180:3.3cm) {$p_5$};
\node[color=black!40] at (120:3.3cm) {$p_6$};

\node[color=black!40] at ($(28:2.6cm)+(0.03cm,-0.13cm)$) {$a_{12}$};
\node[color=black!40] at ($(90:2.5cm)+(-0.0cm,-0.02cm)$) {$a_{61}$};
\node[color=black!40] at ($(152:1.90cm)+(-0.55cm,0.45cm)$) {$a_{56}$};
\node[color=black!40] at ($(208:1.88cm)+(-0.55cm,-0.45cm)$) {$a_{45}$};
\node[color=black!40] at ($(270:2.54cm)+(0cm,0cm)$) {$a_{34}$};
\node[color=black!40] at ($(332:2.62cm)+(-0.01cm,0.082cm)$) {$a_{23}$};

\node[color=black!40] at ($($(0,0)!0.88!(L1)$)+(0.03cm,0.2cm)$) {$l_1$};
\node[color=black!40] at ($($(0,0)!0.88!(L2)$)+(0.13cm,0.12cm)$) {$l_2$};
\node[color=black!40] at ($($(0,0)!0.88!(L3)$)+(0cm,0.0cm)$) {$l_3$};
\node[color=black!40] at ($($(0,0)!0.88!(L4)$)+(0.0cm,0)$) {$l_4$};
\node[color=black!40] at ($($(0,0)!0.88!(L5)$)+(-0.15cm,-0.1cm)$) {$l_5$};
\node[color=black!40] at ($($(0,0)!0.88!(L6)$)+(0.0cm,0.0cm)$) {$l_6$};

\node[color=black!40] at ($($(0,0)!0.88!(R1)$)+(0cm,0.0cm)$) {$r_1$};
\node[color=black!40] at ($($(0,0)!0.88!(R2)$)+(0.18cm,-0.1cm)$) {$r_2$};
\node[color=black!40] at ($($(0,0)!0.88!(R3)$)+(0.0cm,-0.13cm)$) {$r_3$};
\node[color=black!40] at ($($(0,0)!0.88!(R4)$)+(0cm,-0.05cm)$) {$r_4$};
\node[color=black!40] at ($($(0,0)!0.88!(R5)$)+(-0.16cm,0.12cm)$) {$r_5$};
\node[color=black!40] at ($($(0,0)!0.88!(R6)$)+(0.0cm,-0.02cm)$) {$r_6$};

    \draw[color=black!40] (L4) -- (R6);
\draw [black!40] plot [smooth] coordinates {(R2) ($(R2)+(-0.15cm,-0.1cm)$) ($(0,0)!0.6!($0.5*(R2)+0.5*(R4)$)$) ($(R4)+(0.15cm,0.3cm)$) (R4)};
\draw [black!40] plot [smooth] coordinates {(L2) ($(L2)+(-0.36cm,0.15cm)$) ($($(0,0)!0.1!($0.5*(L2)+0.5*(R1)$)$)+(-0.1cm,-0.1cm)$) ($(R1)+(-0.15cm,-0.3cm)$) (R1)};
\draw [black!40] plot [smooth] coordinates {(L1) ($(L1)+(-0.1cm,-0.1cm)$) ($(0,0)!0.6!($0.5*(L1)+0.5*(L6)$)$) ($(L6)+(0.15cm,-0.3cm)$) (L6)};
\draw [black!40] plot [smooth] coordinates {(L5) ($(L5)+(0.3cm,-0.1cm)$) ($(0,0)!0.6!($0.5*(L5)+0.5*(L3)$)$) ($(L3)+(-0.2cm,0.3cm)$) (L3)};
\draw [black!40] plot [smooth] coordinates {(R5) ($(R5)+(0.4cm,0cm)$) ($(0,0)!0.6!($0.5*(R5)+0.5*(R3)$)$) ($(R3)+(-0.2cm,0.2cm)$) (R3)};
\draw[color=black!40, fill=black!40] (P1) circle (1.25pt);
\draw[color=black!40, fill=black!40] (P2) circle (1.25pt);
\draw[color=black!40, fill=black!40] (P3) circle (1.25pt);
\draw[color=black!40, fill=black!40] (P4) circle (1.25pt);
\draw[color=black!40, fill=black!40] (P5) circle (1.25pt);
\draw[color=black!40, fill=black!40] (P6) circle (1.25pt);

\draw[color=black!40, fill=black!40] (A12) circle (1.25pt);
\draw[color=black!40, fill=black!40] (A23) circle (1.25pt);
\draw[color=black!40, fill=black!40] (A34) circle (1.25pt);
\draw[color=black!40, fill=black!40] (A45) circle (1.25pt);
\draw[color=black!40, fill=black!40] (A56) circle (1.25pt);
\draw[color=black!40, fill=black!40] (A61) circle (1.25pt);

\draw[color=black!40, fill=white] (L1) circle (1.25pt);
\draw[color=black!40, fill=white] (L2) circle (1.25pt);
\draw[color=black!40, fill=white] (L3) circle (1.25pt);
\draw[color=black!40, fill=white] (L4) circle (1.25pt);
\draw[color=black!40, fill=white] (L5) circle (1.25pt);
\draw[color=black!40, fill=white] (L6) circle (1.25pt);

\draw[color=black!40, fill=white] (R1) circle (1.25pt);
\draw[color=black!40, fill=white] (R2) circle (1.25pt);
\draw[color=black!40, fill=white] (R3) circle (1.25pt);
\draw[color=black!40, fill=white] (R4) circle (1.25pt);
\draw[color=black!40, fill=white] (R5) circle (1.25pt);
\draw[color=black!40, fill=white] (R6) circle (1.25pt);


\draw[color=black!40, fill=black!40] (-0.925cm,-1.61cm) circle (1.25pt); 
\draw[color=black!40, fill=black!40] (0.795cm,-1.235cm) circle (1.25pt); 
\draw[color=black!40, fill=black!40] (0.660cm,-0.47cm) circle (1.25pt); 
\draw[color=black!40, fill=black!40] (0.286cm,-0.601cm) circle (1.25pt); 
\draw[color=black!40, fill=black!40] (0.905cm,0.72cm) circle (1.25pt); 

    \coordinate (Q1) at (-0.925cm,-1.61cm);
    \coordinate (Q2) at (0.795cm,-1.235cm);
    \coordinate (Q3) at (0.660cm,-0.47cm);
    \coordinate (Q4) at (0.286cm,-0.601cm);
    \coordinate (Q5) at (0.905cm,0.72cm);

    \coordinate (w1) at (barycentric cs:P1=1,L6=1,R2=1);
    \coordinate (w2) at (barycentric cs:R3=1,P2=1,L1=1);
    \coordinate (w3) at (barycentric cs:P3=1,R4=1,L2=1);
    \coordinate (w4) at (barycentric cs:P4=1,R5=1,L3=1);
    \coordinate (w5) at (barycentric cs:P5=1,R6=1,L4=1);
    \coordinate (w6) at (barycentric cs:P6=1,R1=1,L5=1);
    \coordinate (w7) at (barycentric cs:Q2=1,Q3=1,Q4=1);
    \coordinate (w8) at ($(P2)+(0.7cm,0)$);

\draw[fill=blue, color=blue] (w1) circle (1.75pt);
\draw[fill=blue, color=blue] (w2) circle (1.75pt);
\draw[fill=blue, color=blue] (w3) circle (1.75pt);
\draw[fill=blue, color=blue] (w4) circle (1.75pt);
\draw[fill=blue, color=blue] (w5) circle (1.75pt);
\draw[fill=blue, color=blue] (w6) circle (1.75pt);
\draw[fill=blue, color=blue] (w7) circle (1.75pt);
\draw[fill=blue, color=blue] (w8) circle (1.75pt);

    \draw[line width=1.1,  color=blue] (w1) -- (w2);
    \draw[line width=1.1,  color=blue] (w2) -- (w3);
    \draw[line width=1.1,  color=blue] (w3) -- (w4);
    \draw[line width=1.1,  color=blue] (w4) -- (w5);
    \draw[line width=1.1,  color=blue] (w5) -- (w6);
    \draw[line width=1.1,  color=blue] (w6) -- (w1);
    \draw[line width=1.1,  color=blue] (w4) -- (w6);
    \draw[line width=1.1,  color=blue] (w7) -- (w6);
    \draw[line width=1.1,  color=blue] (w7) -- (w2);
    \draw[line width=1.1,  color=blue] (w7) -- (w3);

\draw [line width=1.1,blue] plot [smooth] coordinates {(w2) ($(w2)+(-0.18cm,0.1cm)$) ($(0,0)!0.7!($0.5*(w1)+0.5*(w2)$)$) ($(w1)+(-0.08cm,-0.17cm)$)
 (w1)};

    \draw[line width=1.1, color=blue] (w8) -- (w2);
\draw [line width=1.1,blue] plot [smooth] coordinates {(w8) ($(0,0)!1.25!(A12)$)  (w1)};
\draw [line width=1.1,blue] plot [smooth] coordinates {(w8) ($(0,0)!1.25!(A23)$)  (w3)};
\draw [line width=1.1,blue] plot [smooth] coordinates {(w8) ($(0,0)!1.5!(A12)$)  ($(0,0)!1.1!(P1)$) ($(0,0)!1.3!(A61)$) (w6)};
\draw [line width=1.1,blue] plot [smooth] coordinates {(w8) ($(0,0)!1.5!(A23)$)  ($(0,0)!1.1!(P3)$) ($(0,0)!1.3!(A34)$) (w4)};
\draw [line width=1.1,blue] plot [smooth] coordinates {(w8) ($(0,0)!1.8!(A23)$)  ($(0,0)!1.3!(P3)$) ($(0,0)!1.5!(A34)$) ($(0,0)!1.1!(P4)$) ($(0,0)!1.5!(A45)$) (w5)};
\end{scope}
\node at (0cm,-4cm) {$(a)$};
\node at (8.33cm,-4cm) {$(b)$};
\end{tikzpicture}

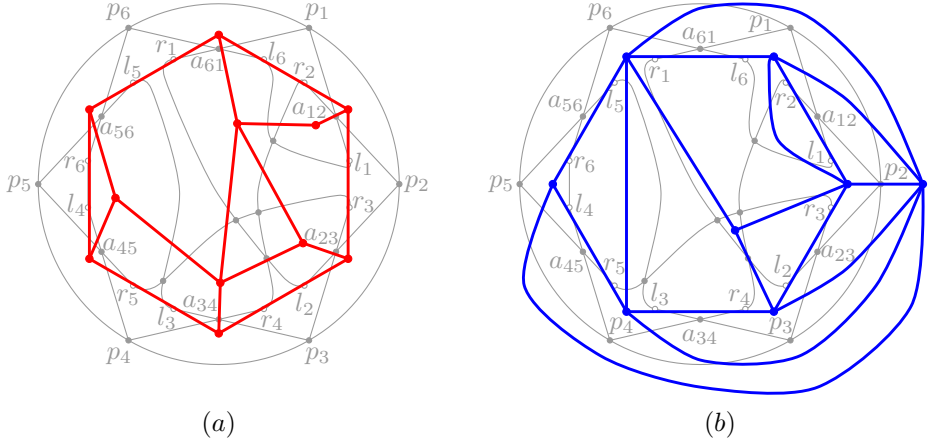
\captionof{figure}{The dual graphs $\mathcal{R}_b(G)$ and $\mathcal{R}_w(G)$}
\end{center}
The graph $\mathcal{R}_b(G)$ contains a pair of faces with two common edges (the case (D)) which corresponds to a double edge in $\mathcal{R}_w(G)$ (the case (B)).  
}\end{exmp}
Now, we show how modify the graph $G$ such that the connections between the points from $\mathcal{O}$ induced by the relation $\sim$ do not change and the graphs $\mathcal{R}_b(G), \mathcal{R}_{w}(G)$ become simple. 

Suppose that $e$ is an edge joining the vertices $v'$ and $v''$ in $\mathcal{R}_b(G)$ or $\mathcal{R}_{w}(G)$ (both the cases are similar). 
If $v'\neq v''$, then we consider the following edges in the same graph ($\mathcal{R}_b(G)$ or $\mathcal{R}_{w}(G)$):

$\bullet$ $e'_+$ and $e'_-$ which occur directly after $e$ in the clockwise and the anticlockwise order on edges incident to $v'$, respectively; 

$\bullet$ $e''_+$ and $e''_-$ which occur directly after $e$ in the clockwise and the anticlockwise order on edges incident to $v''$, respectively; 

\noindent see Fig. 8a. If $v'=v''$, then we exclude the case when the loop $e$ is the boundary of a face (this case will be considered separately). 
The edge $e$ splits the plane into two parts and we consider the following edges:  

$\bullet$ $e'_+$ and $e'_-$ are the edges contained in one of these parts which occur directly after $e$ in the clockwise and the anticlockwise order on edges incident to $v'$, respectively; 

$\bullet$ $e''_+$ and $e''_-$ are the edges contained in the other part of the plane which occur directly after $e$ in the clockwise and the anticlockwise order on edges incident to $v''$, respectively; 

\noindent see Fig. 8b. 
There are precisely two parts of central circuits in $G$ (up to reversing) that pass through $e$ (since $e,e'_\delta,e''_\delta$ are vertices of $G$, where $\delta\in\{+,-\}$): 
$$\dots,e'_+,e,e''_+,\dots\hspace{1cm}\text{ and }\hspace{1cm}\dots,e'_-,e,e''_-,\dots$$
which are marked in blue and red in Fig. 8. 
\begin{center}
\begin{tikzpicture}[scale=1.4]
\begin{scope}[scale=1.1]
\draw[fill=black] (-1cm,0) circle (1.3pt);
\draw[fill=black] (1cm,0) circle (1.3pt);

    \draw (-1cm,0) -- (1cm,0);

    \coordinate (v11) at ($(240:1cm)+(-1cm,0)$);
    \coordinate (v12) at ($(120:1cm)+(-1cm,0)$);
    \coordinate (v21) at ($(60:1cm)+(1cm,0)$);
    \coordinate (v22) at ($(-60:1cm)+(1cm,0)$);

    \draw (-1cm,0) -- (v11);
    \draw (-1cm,0) -- (v12);
    \draw (1cm,0) -- (v21);
    \draw (1cm,0) -- (v22);

\begin{scope}[rotate around={90:($(240:0.67cm)+(-1cm,0)$)}]
   \draw[color=gray] ($($(0,0)!1.2!($(240:0.67cm)+(-1cm,0)$)$)!0.2!(0,0)$) -- ($(0,0)!1.1!($(240:0.67cm)+(-1cm,0)$)$);
\end{scope}
\begin{scope}[rotate around={270:($(240:0.67cm)+(-1cm,0)$)}]
   \draw[color=gray] ($($(0,0)!1.2!($(240:0.67cm)+(-1cm,0)$)$)!0.2!(0,0)$) -- ($(0,0)!1.1!($(240:0.67cm)+(-1cm,0)$)$);
\end{scope}

\begin{scope}[rotate around={90:($(120:0.67cm)+(-1cm,0)$)}]
   \draw[color=gray] ($($(120:0.67cm)+(-1cm,0)$)!0.1!(0,0)$) -- ($(0,0)!1.1!($(120:0.67cm)+(-1cm,0)$)$);
\end{scope}
\begin{scope}[rotate around={270:($(120:0.67cm)+(-1cm,0)$)}]
\end{scope}

\begin{scope}[rotate around={90:($(60:0.67cm)+(1cm,0)$)}]
   \draw[color=gray] ($(60:0.67cm)+(1cm,0)$)-- ($(0,0)!1.1!($(60:0.67cm)+(1cm,0)$)$);
\end{scope}
\begin{scope}[rotate around={270:($(60:0.67cm)+(1cm,0)$)}]
   \draw[color=gray] ($($(0,0)!1.2!($(60:0.67cm)+(1cm,0)$)$)!0.2!(0,0)$) -- ($(0,0)!1.1!($(60:0.67cm)+(1cm,0)$)$);
\end{scope}

\begin{scope}[rotate around={90:($(-60:0.67cm)+(1cm,0)$)}]
   \draw[color=gray] ($(-60:0.67cm)+(1cm,0)$) -- ($(0,0)!1.1!($(-60:0.67cm)+(1cm,0)$)$);
\end{scope}
\begin{scope}[rotate around={270:($(-60:0.67cm)+(1cm,0)$)}]
   \draw[color=gray] ($($(0,0)!1.2!($(-60:0.67cm)+(1cm,0)$)$)!0.2!(0,0)$) -- ($(0,0)!1.1!($(-60:0.67cm)+(1cm,0)$)$);
\end{scope}

    \draw[color=blue] (0,0) -- ($(0,0)!1.2!($(240:0.67cm)+(-1cm,0)$)$);
    \draw[color=red] (0,0) -- ($(0,0)!1.2!($(120:0.67cm)+(-1cm,0)$)$);
    \draw[color=blue] (0,0) -- ($(0,0)!1.2!($(60:0.67cm)+(1cm,0)$)$);
    \draw[color=red] (0,0) -- ($(0,0)!1.2!($(-60:0.67cm)+(1cm,0)$)$);

    \draw (-1cm,0) -- ($(150:0.25cm)+(-1cm,0)$);
    \draw (-1cm,0) -- ($(180:0.25cm)+(-1cm,0)$);
    \draw (-1cm,0) -- ($(210:0.25cm)+(-1cm,0)$);

    \draw (1cm,0) -- ($(30:0.25cm)+(1cm,0)$);
    \draw (1cm,0) -- ($(0:0.25cm)+(1cm,0)$);
    \draw (1cm,0) -- ($(330:0.25cm)+(1cm,0)$);

\node at (-0.95cm,-0.2cm) {$v'$};
\node at (0.88cm,-0.2cm) {$v''$};

\node at (0,-0.2cm) {$e$};

\node at (-1.123cm,-0.695cm) {$e'_+$};
\node at (1.07cm,-0.695cm) {$e''_-$};

\node at (-1.123cm,0.695cm) {$e'_-$};
\node at (1.07cm,0.695cm) {$e''_+$};

\end{scope}

\begin{scope}[xshift=3cm, scale=1.5]

\draw[fill=black] (0,0) circle (0.9pt);

    \draw (0,0) -- (135:0.5cm);
    \draw (0,0) -- (157.5:0.25cm);
    \draw (0,0) -- (180:0.25cm);
    \draw (0,0) -- (202.5:0.25cm);
    \draw (0,0) -- (225.5:0.5cm);

    \draw (0,0) -- (45:0.5cm);
    \draw (0,0) -- (22.5:0.25cm);
    \draw (0,0) -- (0:0.25cm);
    \draw (0,0) -- (337.5:0.25cm);
    \draw (0,0) -- (315.5:0.5cm);

\draw (0.75cm,0) circle (0.75cm);

    \coordinate (e2p) at ($(0,0)!0.7!(135:0.5cm)$);
    \coordinate (e1p) at ($(0,0)!0.7!(225:0.5cm)$);
    \coordinate (e1b) at ($(0,0)!0.7!(45:0.5cm)$);
    \coordinate (e2b) at ($(0,0)!0.7!(315:0.5cm)$);
    \coordinate (e) at (1.5cm,0);

    \draw[color=gray] ($(e2b)+(-0.02cm,0.08cm)$) -- ($(e2b)+(0.02cm,-0.08cm)$);
    \draw[color=gray] ($(e2p)+(-0.02cm,0.08cm)$) -- ($(e2p)+(0.02cm,-0.08cm)$);

\draw [red] plot [smooth] coordinates {($(e2p)+(225:0.08cm)$) (e2p) ($(120:1cm)+(0.75cm,0)$) ($(96:0.95cm)+(0.75cm,0)$) ($(72:0.9cm)+(0.75cm,0)$) ($(48:0.85cm)+(0.75cm,0)$) ($(24:0.85cm)+(0.75cm,0)$) (e) ($(-48:0.65cm)+(0.75cm,0)$) ($(-72:0.65cm)+(0.75cm,0)$) ($(-96:0.65cm)+(0.75cm,0)$) ($(-120:0.65cm)+(0.75cm,0)$) ($(-144:0.65cm)+(0.75cm,0)$) (e2b) ($(e2b)+(45:0.08cm)$)};

    \draw[color=gray] ($(e1b)+(0.02cm,0.08cm)$) -- ($(e1b)+(-0.02cm,-0.08cm)$);
    \draw[color=gray] ($(e1p)+(0.02cm,0.08cm)$) -- ($(e1p)+(-0.02cm,-0.08cm)$);

\draw [blue] plot [smooth] coordinates {($(e1p)+(-225:0.08cm)$) (e1p) ($(-120:1cm)+(0.75cm,0)$) ($(-96:0.95cm)+(0.75cm,0)$) ($(-72:0.9cm)+(0.75cm,0)$) ($(-48:0.85cm)+(0.75cm,0)$) ($(-24:0.85cm)+(0.75cm,0)$) (e) ($(48:0.65cm)+(0.75cm,0)$) ($(72:0.65cm)+(0.75cm,0)$) ($(96:0.65cm)+(0.75cm,0)$) ($(120:0.65cm)+(0.75cm,0)$) ($(144:0.65cm)+(0.75cm,0)$) (e1b) ($(e1b)+(-45:0.08cm)$)};


\node at ($($(0,0)!1!(120:0.55cm)$)+(0.015cm,0.03cm)$) {$e'_-$};
\node at ($($(0,0)!1!(-120:0.55cm)$)+(0cm,0cm)$) {$e'_+$};

\node at ($($(0,0)!1!(60:0.55cm)$)+(0.19cm,-0.2cm)$) {$e''_+$};
\node at ($($(0,0)!1!(-60:0.55cm)$)+(0.19cm,0.2cm)$) {$e''_-$};

\node at (1.6cm,0) {$e$};

\node at (-0.05cm,0.2cm) {$v'$};

\end{scope}


\node at (0cm,-1.75cm) {$(a)$};

\node at (4cm,-1.75cm) {$(b)$};

\end{tikzpicture}

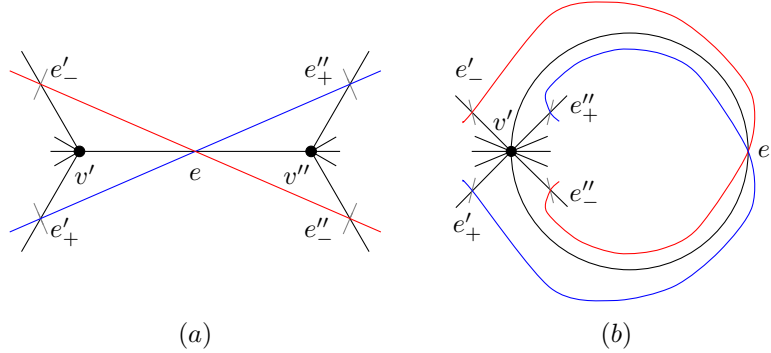
\captionof{figure}{Parts of central circuits}
\end{center}

\noindent For each of these cases we replace $e$ by the graphs presented in Fig. 9a and Fig. 9b, respectively. 

\begin{center}
\begin{tikzpicture}[scale=1.5]
\begin{scope}
\draw[fill=black] (-1cm,0) circle (1.3pt);
\draw[fill=black] (1cm,0) circle (1.3pt);


    \coordinate (v11) at ($(240:1cm)+(-1cm,0)$);
    \coordinate (v12) at ($(120:1cm)+(-1cm,0)$);
    \coordinate (v21) at ($(60:1cm)+(1cm,0)$);
    \coordinate (v22) at ($(-60:1cm)+(1cm,0)$);

    \draw (-1cm,0) -- (v11);
    \draw (-1cm,0) -- (v12);
    \draw (1cm,0) -- (v21);
    \draw (1cm,0) -- (v22);

\draw[fill=black] (-0.33cm,0.33cm) circle (1.3pt);
\draw[fill=black] (-0.33cm,-0.33cm) circle (1.3pt);
\draw[fill=black] (0.33cm,0.33cm) circle (1.3pt);
\draw[fill=black] (0.33cm,-0.33cm) circle (1.3pt);

\draw[fill=black] (0cm,1cm) circle (1.3pt);

    \draw (-0.33cm,0.33cm) -- (-0.33cm,-0.33cm) -- (0.33cm,-0.33cm) -- (0.33cm,0.33cm) -- cycle;
    \draw (-0.33cm,0.33cm) -- (0cm,1cm) -- (0.33cm,0.33cm); 
    \draw (-0.33cm,0.33cm) -- (-1cm,0cm) -- (-0.33cm,-0.33cm); 
    \draw (0.33cm,0.33cm) -- (1cm,0cm) -- (0.33cm,-0.33cm); 
    \draw (0cm,1cm) -- (1cm,0cm);
\draw plot [smooth] coordinates {(-0.33cm,-0.33cm) ($($(-0.33cm,-0.33cm)!0.2!(1cm,0cm)$)+(0.3cm,-0.5cm)$) ($($(-0.33cm,-0.33cm)!0.5!(1cm,0cm)$)+(0.4cm,-0.5cm)$) (1cm,0cm)};

    \coordinate (x) at ($(-1cm,0)!0.5!(-0.33cm,0.33cm)$);
    \coordinate (u) at ($(-1cm,0)!0.5!(-0.33cm,-0.33cm)$);
    \coordinate (p) at ($(-0.33cm,0.33cm)!0.5!(-0.33cm,-0.33cm)$);

    \coordinate (k) at ($(1cm,0)!0.5!(0.33cm,0.33cm)$);
    \coordinate (m) at ($(1cm,0)!0.5!(0.33cm,-0.33cm)$);
    \coordinate (w) at ($(0.33cm,0.33cm)!0.5!(0.33cm,-0.33cm)$);

    \coordinate (z) at ($(1cm,0)!0.5!(0,1cm)$);

    \coordinate (y) at ($(-0.33cm,0.33cm)!0.5!(0,1cm)$);
    \coordinate (n) at ($(0.33cm,0.33cm)!0.5!(0,1cm)$);
    \coordinate (r) at ($(-0.33cm,0.33cm)!0.5!(0.33cm,0.33cm)$);

    \coordinate (s) at ($(-0.33cm,-0.33cm)!0.5!(0.33cm,-0.33cm)$);

    \coordinate (t) at (0.5cm,-0.77cm);

\begin{scope}[rotate around={90:($(240:0.67cm)+(-1cm,0)$)}]
   \draw[color=gray] ($($(0,0)!1.2!($(240:0.67cm)+(-1cm,0)$)$)!0.2!(0,0)$) -- ($(0,0)!1.1!($(240:0.67cm)+(-1cm,0)$)$);
\end{scope}
\begin{scope}[rotate around={270:($(240:0.67cm)+(-1cm,0)$)}]
   \draw[color=gray] ($($(0,0)!1.2!($(240:0.67cm)+(-1cm,0)$)$)!0.2!(0,0)$) -- ($(0,0)!1.1!($(240:0.67cm)+(-1cm,0)$)$);
\end{scope}

\begin{scope}[rotate around={90:($(120:0.67cm)+(-1cm,0)$)}]
   \draw[color=gray] ($($(120:0.67cm)+(-1cm,0)$)!0.1!(0,0)$) -- ($(0,0)!1.1!($(120:0.67cm)+(-1cm,0)$)$);
\end{scope}
\begin{scope}[rotate around={270:($(120:0.67cm)+(-1cm,0)$)}]
\end{scope}

\begin{scope}[rotate around={90:($(60:0.67cm)+(1cm,0)$)}]
   \draw[color=gray] ($(60:0.67cm)+(1cm,0)$)-- ($(0,0)!1.1!($(60:0.67cm)+(1cm,0)$)$);
\end{scope}
\begin{scope}[rotate around={270:($(60:0.67cm)+(1cm,0)$)}]
   \draw[color=gray] ($($(0,0)!1.2!($(60:0.67cm)+(1cm,0)$)$)!0.2!(0,0)$) -- ($(0,0)!1.1!($(60:0.67cm)+(1cm,0)$)$);
\end{scope}

\begin{scope}[rotate around={90:($(-60:0.67cm)+(1cm,0)$)}]
   \draw[color=gray] ($(-60:0.67cm)+(1cm,0)$) -- ($(0,0)!1.1!($(-60:0.67cm)+(1cm,0)$)$);
\end{scope}
\begin{scope}[rotate around={270:($(-60:0.67cm)+(1cm,0)$)}]
   \draw[color=gray] ($($(0,0)!1.2!($(-60:0.67cm)+(1cm,0)$)$)!0.2!(0,0)$) -- ($(0,0)!1.1!($(-60:0.67cm)+(1cm,0)$)$);
\end{scope}


\draw [blue] plot [smooth] coordinates {($(240:0.67cm)+(-1cm,0)$) (u) (p) (r) (n) (z) ($(60:0.67cm)+(1cm,0)$)};

\draw [red] plot [smooth] coordinates {($(120:0.67cm)+(-1cm,0)$) (x) (p) (s) ($($0.5*(s)+0.5*(m)$)+(0.1cm,-0.2cm)$) (m) (k) (n) (y) (x) (u) ($($0.5*(u)+0.5*(t)$)+(0.15cm,-0.35cm)$) (t) (m) (w) (r) (y) ($($0.5*(y)+0.5*(z)$)+(-0.25cm,0.6cm)$) ($($0.5*(y)+0.5*(z)$)+(0.15cm,0.35cm)$) (z) (k) (w) (s) (t) ($(-60:0.67cm)+(1cm,0)$)};

\begin{scope}[rotate around={0:($(240:0.67cm)+(-1cm,0)$)}]
   \draw[color=blue] ($(240:0.67cm)+(-1cm,0)$) -- ($(0,0)!1.1!($(240:0.67cm)+(-1cm,0)$)$);
\end{scope}
\begin{scope}[rotate around={0:($(120:0.67cm)+(-1cm,0)$)}]
   \draw[color=red] ($(120:0.67cm)+(-1cm,0)$) -- ($(0,0)!1.1!($(120:0.67cm)+(-1cm,0)$)$);
\end{scope}

\begin{scope}[rotate around={-10:($(60:0.67cm)+(1cm,0)$)}]
   \draw[color=blue] ($(60:0.67cm)+(1cm,0)$)-- ($(0,0)!1.1!($(60:0.67cm)+(1cm,0)$)$);
\end{scope}
\begin{scope}[rotate around={30:($(-60:0.67cm)+(1cm,0)$)}]
   \draw[color=red] ($(-60:0.67cm)+(1cm,0)$)-- ($(0,0)!1.1!($(-60:0.67cm)+(1cm,0)$)$);
\end{scope}

    \draw (-1cm,0) -- ($(150:0.25cm)+(-1cm,0)$);
    \draw (-1cm,0) -- ($(180:0.25cm)+(-1cm,0)$);
    \draw (-1cm,0) -- ($(210:0.25cm)+(-1cm,0)$);

    \draw (1cm,0) -- ($(30:0.25cm)+(1cm,0)$);
    \draw (1cm,0) -- ($(0:0.25cm)+(1cm,0)$);
    \draw (1cm,0) -- ($(330:0.25cm)+(1cm,0)$);

\node at (-0.95cm,0.18cm) {$v'$};
\node at (1cm,0.23cm) {$v''$};

\node at (-1.123cm,-0.665cm) {$e'_+$};
\node at (1.13cm,-0.83cm) {$e''_-$};

\node at (-1.123cm,0.695cm) {$e'_-$};
\node at (1.13cm,0.73cm) {$e''_+$};

\end{scope}

\begin{scope}[xshift=3cm, scale=1.3]

\draw[fill=black] (0,0) circle (0.9pt);

    \draw (0,0) -- (135:0.5cm);
    \draw (0,0) -- (157.5:0.25cm);
    \draw (0,0) -- (180:0.25cm);
    \draw (0,0) -- (202.5:0.25cm);
    \draw (0,0) -- (225.5:0.5cm);

    \draw (0,0) -- (45:0.5cm);
    \draw (0,0) -- (22.5:0.25cm);
    \draw (0,0) -- (0:0.25cm);
    \draw (0,0) -- (337.5:0.25cm);
    \draw (0,0) -- (315.5:0.5cm);

\draw (0.75cm,0) circle (0.75cm);

    \coordinate (e2p) at ($(0,0)!0.7!(135:0.5cm)$);
    \coordinate (e1p) at ($(0,0)!0.7!(225:0.5cm)$);
    \coordinate (e1b) at ($(0,0)!0.7!(45:0.5cm)$);
    \coordinate (e2b) at ($(0,0)!0.7!(315:0.5cm)$);
    \coordinate (e) at (1.5cm,0);

    \coordinate (w1) at (0.75cm,0.75cm);
    \coordinate (w2) at ($0.5*(1.38cm,0.4cm)+0.5*(1.9cm,0)$);
    \coordinate (w9) at ($0.5*(1.38cm,0.4cm)+0.5*(1.1cm,0)$);
    \coordinate (w3) at ($($(w2)!0.43!(w9)$)+(0,-0.02cm)$);
    \coordinate (w4) at ($0.5*(1.5cm,0)+0.5*(1.1cm,0)$);
    \coordinate (w5) at ($0.5*(1.5cm,0)+0.5*(1.9cm,0)$);
    \coordinate (w6) at (1.08cm,-0.37cm);
    \coordinate (w8) at (1.705cm,-0.37cm);
    \coordinate (w7) at (1.23cm,-0.58cm);



    \draw[color=gray] ($(e2b)+(-0.02cm,0.08cm)$) -- ($(e2b)+(0.02cm,-0.08cm)$);
    \draw[color=gray] ($(e2p)+(-0.02cm,0.08cm)$) -- ($(e2p)+(0.02cm,-0.08cm)$);


    \draw[color=gray] ($(e1b)+(0.02cm,0.08cm)$) -- ($(e1b)+(-0.02cm,-0.08cm)$);
    \draw[color=gray] ($(e1p)+(0.02cm,0.08cm)$) -- ($(e1p)+(-0.02cm,-0.08cm)$);

\draw [red] plot [smooth] coordinates {($(e2p)+(225:0.08cm)$) (e2p) ($(150:0.94cm)+(0.75cm,0)$) ($(135:0.87cm)+(0.75cm,0)$) (w1) ($(45:0.55cm)+(0.75cm,0)$) (w9) (w4) (w7) (w8) ($(-15:1.12cm)+(0.75cm,0)$) ($(0:1.3cm)+(0.75cm,0)$) ($(10:1.2cm)+(0.75cm,0)$) (w2) (w3) (w4) (w6) 
($(280:0.55cm)+(0.75cm,0)$) ($(250:0.57cm)+(0.75cm,0)$) ($(235:0.59cm)+(0.75cm,0)$) ($(220:0.62cm)+(0.75cm,0)$) ($(210:0.62cm)+(0.75cm,0)$) 
(e2b) ($(e2b)+(45:0.08cm)$)};

\draw [blue] plot [smooth] coordinates {($(e1p)+(-225:0.08cm)$) (e1p) 
($(230:0.96cm)+(0.75cm,0)$) ($(250:0.96cm)+(0.75cm,0)$) ($(270:0.96cm)+(0.75cm,0)$) ($(290:0.99cm)+(0.75cm,0)$) ($(310:1.07cm)+(0.75cm,0)$) ($(330:1.07cm)+(0.75cm,0)$) 
(w8) (w5) (w3) (w9) 
($($0.5*(w9)+0.5*(w6)$)+(-0.2cm,0.1cm)$)
(w6) (w7) 
($($0.5*(w7)+0.5*(w5)$)+(0.1cm,0.0cm)$)
(w5) (w2) 
($($0.5*(w2)+0.5*(w1)$)+(0.24cm,0.135cm)$)
(w1) 
($($0.5*(w1)+0.5*(e1b)$)+(-0.1cm,0.08cm)$)
($($0.5*(w1)+0.5*(e1b)$)+(-0.25cm,-0.1cm)$)
(e1b) ($(e1b)+(-45:0.08cm)$)};


\node at ($($(0,0)!1!(120:0.55cm)$)+(0.015cm,0.03cm)$) {$e'_-$};
\node at ($($(0,0)!1!(-120:0.55cm)$)+(0cm,0cm)$) {$e'_+$};

\node at ($($(0,0)!1!(60:0.55cm)$)+(0.19cm,-0.2cm)$) {$e''_+$};
\node at ($($(0,0)!1!(-60:0.55cm)$)+(0.19cm,0.2cm)$) {$e''_-$};


\node at (-0.05cm,0.2cm) {$v'$};

\draw[fill=black] (1.5cm,0) circle (0.9pt);
\draw[fill=black] (1.1cm,0) circle (0.9pt);
\draw[fill=black] (1.9cm,0) circle (0.9pt);

\draw[fill=black] (1.38cm,0.4cm) circle (0.9pt);

    \draw (1.1cm,0) -- (1.5cm,0) -- (1.9cm,0);
    \draw (1.1cm,0) -- (1.38cm,0.4cm) -- (1.9cm,0);

\draw plot [smooth] coordinates {(0,0) ($(195:0.72cm)+(0.75cm,0)$) ($(210:0.7cm)+(0.75cm,0)$) ($(240:0.68cm)+(0.75cm,0)$) ($(270:0.65cm)+(0.75cm,0)$) ($(300:0.55cm)+(0.75cm,0)$) ($(330:0.45cm)+(0.75cm,0)$) (1.1cm,0)};
\draw plot [smooth] coordinates {(0,0) ($(195:0.77cm)+(0.75cm,0)$) ($(210:0.8cm)+(0.75cm,0)$) ($(225:0.81cm)+(0.75cm,0)$) ($(240:0.82cm)+(0.75cm,0)$) ($(255:0.835cm)+(0.75cm,0)$) ($(270:0.85cm)+(0.75cm,0)$) ($(285:0.9cm)+(0.75cm,0)$) ($(300:0.95cm)+(0.75cm,0)$) ($(315:1.00cm)+(0.75cm,0)$) ($(330:1.02cm)+(0.75cm,0)$) (1.9cm,0)};

\end{scope}


\node at (0cm,-1.5cm) {$(a)$};

\node at (4cm,-1.5cm) {$(b)$};
\end{tikzpicture}

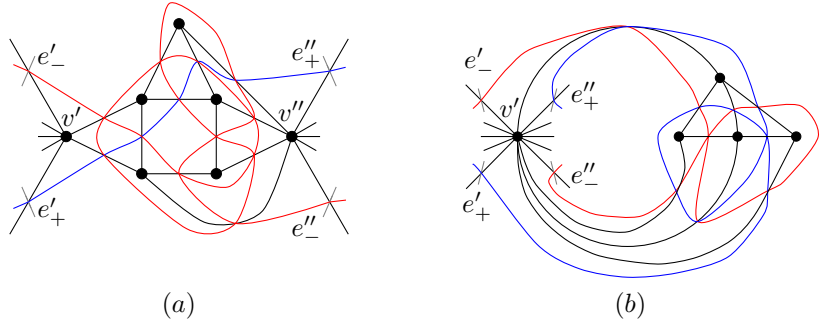
\captionof{figure}{Two types of expansion}
\end{center}

\noindent This operation will be called the {\it expansion} of $e$.
It replaces the mentioned above parts of central circuits by 
$$\dots,e'_+,E_+,e''_+,\dots\hspace{1cm}\text{ and }\hspace{1cm}\dots,e'_-,E_-,e''_-,\dots$$
(respectively), where $E_+, E_-$ are two intersecting trails in $G$ and each $E_\delta$ joins $e'_\delta$ and $e''_\delta$ for $\delta\in\{+,-\}$ (they are marked in blue and red in Fig. 9). 
Thus, central circuits do not change in a significant way. 

Now, we explain how transform $\mathcal{R}_b(G), \mathcal{R}_{w}(G)$ to simple graphs if at least one of the possibilities (A)--(D) is realized.  
Without loss of generality we can consider $\mathcal{R}_b(G)$. 
Furthermore, we restrict ourselves to the cases (A) and (B) (since (C) and (D) correspond to (A) and (B), respectively, in the dual graph). 
The case (A) will be decomposed in two subcases. 

(A1). Suppose that $\mathcal{R}_{b}(G)$ contains a face whose boundary is a loop.
The corresponding parts of mutually reversed central circuits from $G$ are also loops. 
The loops can be removed from these graphs without changing the central circuits in a significant way (see Fig. 10).  
\begin{center}
\begin{tikzpicture}[scale=1.2]
\begin{scope}
\draw[fill=black] (0,0) circle (1.7pt);

    \draw (0,-1.5cm) -- (0,1.5cm);
    \draw (0,0) -- (120:0.5cm);
    \draw (0,0) -- (150:0.5cm);
    \draw (0,0) -- (180:0.5cm);
    \draw (0,0) -- (210:0.5cm);
    \draw (0,0) -- (240:0.5cm);

\draw (1cm,0) circle (1cm);

    \draw[color=gray] (100:0.55cm) -- ($(100:0.55cm)!2!(0,0.75cm)$);

    \draw[color=gray] (260:0.55cm) -- ($(260:0.55cm)!2!(0,-0.75cm)$);

\draw [red] plot [smooth] coordinates {(0,0.75cm) ($(90:1.25cm)+(1cm,0)$) ($(45:1.25cm)+(1cm,0)$) (2cm,0) 
($(-45:0.3cm)+(1.5cm,0)$) ($(-90:0.25cm)+(1.5cm,0)$) ($(-135:0.3cm)+(1.5cm,0)$)
 (1.1cm,0) 
($(135:0.3cm)+(1.5cm,0)$) ($(90:0.25cm)+(1.5cm,0)$) ($(45:0.3cm)+(1.5cm,0)$) 
(2cm,0) ($(-45:1.25cm)+(1cm,0)$) ($(-90:1.25cm)+(1cm,0)$) (0,-0.75cm)};

    \draw[color=red] (0,-0.75cm) -- ($($(-90:1.25cm)+(1cm,0)$)!1.2!(0,-0.75cm)$);
    \draw[color=red] (0,0.75cm) -- ($($(90:1.25cm)+(1cm,0)$)!1.2!(0,0.75cm)$);

\end{scope}

\begin{scope}[xshift=3cm]
    \draw[->] (-0.5cm,0) -- (1cm,0);
\end{scope}

\begin{scope}[xshift=5cm]
\draw[fill=black] (0,0) circle (1.7pt);

    \draw (0,-1.5cm) -- (0,1.5cm);
    \draw (0,0) -- (120:0.5cm);
    \draw (0,0) -- (150:0.5cm);
    \draw (0,0) -- (180:0.5cm);
    \draw (0,0) -- (210:0.5cm);
    \draw (0,0) -- (240:0.5cm);


    \draw[color=gray] (100:0.55cm) -- ($(100:0.55cm)!2!(0,0.75cm)$);

    \draw[color=gray] (260:0.55cm) -- ($(260:0.55cm)!2!(0,-0.75cm)$);

\draw [red] plot [smooth] coordinates {(0,0.75cm) ($(90:1cm)+(1cm,0)$) ($(45:1cm)+(1cm,0)$) 
(2cm,0) ($(-45:1cm)+(1cm,0)$) ($(-90:1cm)+(1cm,0)$) (0,-0.75cm)};

    \draw[color=red] (0,-0.75cm) -- ($($(-90:1.25cm)+(1cm,0)$)!1.2!(0,-0.75cm)$);
    \draw[color=red] (0,0.75cm) -- ($($(90:1.25cm)+(1cm,0)$)!1.2!(0,0.75cm)$);

\end{scope}

\end{tikzpicture}

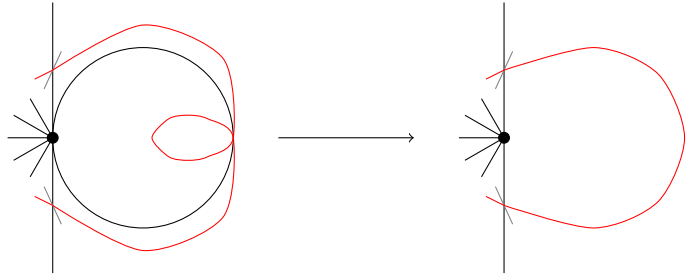
\captionof{figure}{Removing a loop}
\end{center}

(A2). If $e$ is a loop in $\mathcal{R}_b(G)$ which is not a boundary of a face, then we use the expansion to $e$. 

(B). If two distinct vertices are connected by $m\geq 2$ edges, then we expand any $m-1$ of them. 

\begin{exmp}\label{ex5}{\rm 
Since $\mathcal{R}_w(G)$ from Example \ref{ex4} contains two edges connecting the same pair of vertices (the case (C)), we expand one of these edges, see Fig. 11. 
\begin{center}
\begin{tikzpicture}[scale=0.6]
\draw [color=black, decoration={markings,
mark=at position 1 with {\arrow[scale=1.5,>=stealth]{>}}},
postaction={decorate},xshift=6cm] (-7.5cm,0) -- (-4.3cm,0);
\begin{scope}[xshift=-6cm]
    \coordinate (P1) at (60:3cm);
    \coordinate (P2) at (0:3cm);
    \coordinate (P3) at (300:3cm);
    \coordinate (P4) at (240:3cm);
    \coordinate (P5) at (180:3cm);
    \coordinate (P6) at (120:3cm);

    \coordinate (A12) at (30:2.25cm);
    \coordinate (A23) at (330:2.25cm);
    \coordinate (A34) at (270:2.25cm);
    \coordinate (A45) at (210:2.25cm);
    \coordinate (A56) at (150:2.25cm);
    \coordinate (A61) at (90:2.25cm);

    \coordinate (L1) at ($(P1)!1.5!(A12)$);
    \coordinate (L2) at ($(P2)!1.5!(A23)$);
    \coordinate (L3) at ($(P3)!1.5!(A34)$);
    \coordinate (L4) at ($(P4)!1.5!(A45)$);
    \coordinate (L5) at ($(P5)!1.5!(A56)$);
    \coordinate (L6) at ($(P6)!1.5!(A61)$);

    \coordinate (R1) at ($(P1)!1.5!(A61)$);
    \coordinate (R2) at ($(P2)!1.5!(A12)$);
    \coordinate (R3) at ($(P3)!1.5!(A23)$);
    \coordinate (R4) at ($(P4)!1.5!(A34)$);
    \coordinate (R5) at ($(P5)!1.5!(A45)$);
    \coordinate (R6) at ($(P6)!1.5!(A56)$);

    \coordinate (Q1) at (-0.925cm,-1.61cm);
    \coordinate (Q2) at (0.795cm,-1.235cm);
    \coordinate (Q3) at (0.660cm,-0.47cm);
    \coordinate (Q4) at (0.286cm,-0.601cm);
    \coordinate (Q5) at (0.905cm,0.72cm);

    \coordinate (w1) at (barycentric cs:P1=1,L6=1,R2=1);
    \coordinate (w2) at (barycentric cs:R3=1,P2=1,L1=1);
    \coordinate (w3) at (barycentric cs:P3=1,R4=1,L2=1);
    \coordinate (w4) at (barycentric cs:P4=1,R5=1,L3=1);
    \coordinate (w5) at (barycentric cs:P5=1,R6=1,L4=1);
    \coordinate (w6) at (barycentric cs:P6=1,R1=1,L5=1);
    \coordinate (w7) at (barycentric cs:Q2=1,Q3=1,Q4=1);
    \coordinate (w8) at ($(P2)+(0.7cm,0)$);

\draw[fill=black, color=black] (w1) circle (1.75pt);
\draw[fill=black, color=black] (w2) circle (1.75pt);
\draw[fill=black, color=black] (w3) circle (1.75pt);
\draw[fill=black, color=black] (w4) circle (1.75pt);
\draw[fill=black, color=black] (w5) circle (1.75pt);
\draw[fill=black, color=black] (w6) circle (1.75pt);
\draw[fill=black, color=black] (w7) circle (1.75pt);
\draw[fill=black, color=black] (w8) circle (1.75pt);

    \draw[color=black] (w1) -- (w2);
    \draw[color=black] (w2) -- (w3);
    \draw[color=black] (w3) -- (w4);
    \draw[color=black] (w4) -- (w5);
    \draw[color=black] (w5) -- (w6);
    \draw[color=black] (w6) -- (w1);
    \draw[color=black] (w4) -- (w6);
    \draw[color=black] (w7) -- (w6);
    \draw[color=black] (w7) -- (w2);
    \draw[color=black] (w7) -- (w3);

\draw [black] plot [smooth] coordinates {(w2) ($(w2)+(-0.18cm,0.1cm)$) ($(0,0)!0.7!($0.5*(w1)+0.5*(w2)$)$) ($(w1)+(-0.08cm,-0.17cm)$)
 (w1)};

    \draw[color=black] (w8) -- (w2);
\draw [black] plot [smooth] coordinates {(w8) ($(0,0)!1.25!(A12)$)  (w1)};
\draw [black] plot [smooth] coordinates {(w8) ($(0,0)!1.25!(A23)$)  (w3)};
\draw [black] plot [smooth] coordinates {(w8) ($(0,0)!1.5!(A12)$)  ($(0,0)!1.1!(P1)$) ($(0,0)!1.3!(A61)$) (w6)};
\draw [black] plot [smooth] coordinates {(w8) ($(0,0)!1.5!(A23)$)  ($(0,0)!1.1!(P3)$) ($(0,0)!1.3!(A34)$) (w4)};
\draw [black] plot [smooth] coordinates {(w8) ($(0,0)!1.8!(A23)$)  ($(0,0)!1.3!(P3)$) ($(0,0)!1.5!(A34)$) ($(0,0)!1.1!(P4)$) ($(0,0)!1.5!(A45)$) (w5)};
\end{scope}
\begin{scope}[xshift=5cm]
    \coordinate (P1) at (60:3cm);
    \coordinate (P2) at (0:3cm);
    \coordinate (P3) at (300:3cm);
    \coordinate (P4) at (240:3cm);
    \coordinate (P5) at (180:3cm);
    \coordinate (P6) at (120:3cm);

    \coordinate (A12) at (30:2.25cm);
    \coordinate (A23) at (330:2.25cm);
    \coordinate (A34) at (270:2.25cm);
    \coordinate (A45) at (210:2.25cm);
    \coordinate (A56) at (150:2.25cm);
    \coordinate (A61) at (90:2.25cm);

    \coordinate (L1) at ($(P1)!1.5!(A12)$);
    \coordinate (L2) at ($(P2)!1.5!(A23)$);
    \coordinate (L3) at ($(P3)!1.5!(A34)$);
    \coordinate (L4) at ($(P4)!1.5!(A45)$);
    \coordinate (L5) at ($(P5)!1.5!(A56)$);
    \coordinate (L6) at ($(P6)!1.5!(A61)$);

    \coordinate (R1) at ($(P1)!1.5!(A61)$);
    \coordinate (R2) at ($(P2)!1.5!(A12)$);
    \coordinate (R3) at ($(P3)!1.5!(A23)$);
    \coordinate (R4) at ($(P4)!1.5!(A34)$);
    \coordinate (R5) at ($(P5)!1.5!(A45)$);
    \coordinate (R6) at ($(P6)!1.5!(A56)$);


    \coordinate (Q1) at (-0.925cm,-1.61cm);
    \coordinate (Q2) at (0.795cm,-1.235cm);
    \coordinate (Q3) at (0.660cm,-0.47cm);
    \coordinate (Q4) at (0.286cm,-0.601cm);
    \coordinate (Q5) at (0.905cm,0.72cm);

    \coordinate (w1) at (barycentric cs:P1=1,L6=1,R2=1);
    \coordinate (w2) at (barycentric cs:R3=1,P2=1,L1=1);
    \coordinate (w3) at (barycentric cs:P3=1,R4=1,L2=1);
    \coordinate (w4) at (barycentric cs:P4=1,R5=1,L3=1);
    \coordinate (w5) at (barycentric cs:P5=1,R6=1,L4=1);
    \coordinate (w6) at (barycentric cs:P6=1,R1=1,L5=1);
    \coordinate (w7) at (barycentric cs:Q2=1,Q3=1,Q4=1);
    \coordinate (w8) at ($(P2)+(0.7cm,0)$);

\draw[fill=black, color=black] (w1) circle (1.75pt);
\draw[fill=black, color=black] (w2) circle (1.75pt);
\draw[fill=black, color=black] (w3) circle (1.75pt);
\draw[fill=black, color=black] (w4) circle (1.75pt);
\draw[fill=black, color=black] (w5) circle (1.75pt);
\draw[fill=black, color=black] (w6) circle (1.75pt);
\draw[fill=black, color=black] (w7) circle (1.75pt);
\draw[fill=black, color=black] (w8) circle (1.75pt);

    \draw[color=black] (w1) -- (w2);
    \draw[color=black] (w2) -- (w3);
    \draw[color=black] (w3) -- (w4);
    \draw[color=black] (w4) -- (w5);
    \draw[color=black] (w5) -- (w6);
    \draw[color=black] (w6) -- (w1);
    \draw[color=black] (w4) -- (w6);
    \draw[color=black] (w7) -- (w6);
    \draw[color=black] (w7) -- (w2);
    \draw[color=black] (w7) -- (w3);


    \draw[color=black] (w8) -- (w2);
\draw [black] plot [smooth] coordinates {(w8) ($(0,0)!1.25!(A12)$)  (w1)};
\draw [black] plot [smooth] coordinates {(w8) ($(0,0)!1.25!(A23)$)  (w3)};
\draw [black] plot [smooth] coordinates {(w8) ($(0,0)!1.5!(A12)$)  ($(0,0)!1.1!(P1)$) ($(0,0)!1.3!(A61)$) (w6)};
\draw [black] plot [smooth] coordinates {(w8) ($(0,0)!1.5!(A23)$)  ($(0,0)!1.1!(P3)$) ($(0,0)!1.3!(A34)$) (w4)};
\draw [black] plot [smooth] coordinates {(w8) ($(0,0)!1.8!(A23)$)  ($(0,0)!1.3!(P3)$) ($(0,0)!1.5!(A34)$) ($(0,0)!1.1!(P4)$) ($(0,0)!1.5!(A45)$) (w5)};

    \coordinate (z1) at (barycentric cs:w1=0.6,w6=1.6,w2=0.6,w7=1.8);
\draw[fill=black, color=black] (z1) circle (1.75pt);
    \draw[color=black] (z1) -- (w1);
    \coordinate (z2) at (barycentric cs:w1=0.8,w2=0.7,z1=1.8);
\draw[fill=black, color=black] (z2) circle (1.75pt);

    \coordinate (z3) at (barycentric cs:w1=0.1,w2=1.4,z1=2.3);
\draw[fill=black, color=black] (z3) circle (1.75pt);

    \coordinate (z4) at (barycentric cs:z2=1.8,w1=1,w2=1);
\draw[fill=black, color=black] (z4) circle (1.75pt);

    \coordinate (z5) at (barycentric cs:z3=1.8,w1=1,w2=1.9);
\draw[fill=black, color=black] (z5) circle (1.75pt);

    \draw[color=black] (z1) -- (z2);
    \draw[color=black] (z2) -- (w1);
    \draw[color=black] (z2) -- (z3);
    \draw[color=black] (z1) -- (z3);
    \draw[color=black] (w2) -- (z3);

    \draw[color=black] (z4) -- (w1);
    \draw[color=black] (z4) -- (z2);
    \draw[color=black] (z4) -- (z5);
    \draw[color=black] (z5) -- (z3);
    \draw[color=black] (z5) -- (w2);
\end{scope}
\end{tikzpicture}

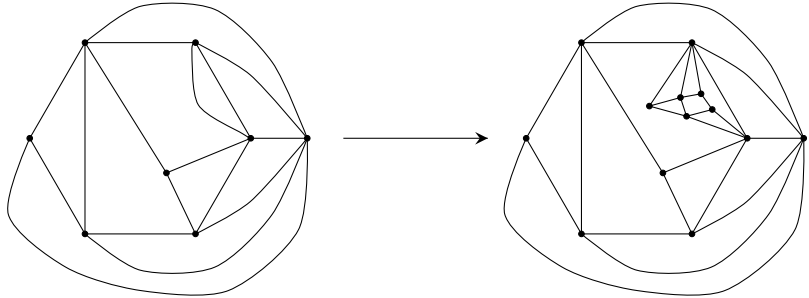
\captionof{figure}{The expansion of an edge in $\mathcal{R}_w(G)$}
\end{center}
This simultaneously modify $\mathcal{R}_b(G)$ and we obtain a graph without the possibility (A), see Fig. 12. 
\begin{center}
\begin{tikzpicture}[scale=0.6]
\draw [color=black, decoration={markings,
mark=at position 1 with {\arrow[scale=1.5,>=stealth]{>}}},
postaction={decorate}] (-8cm,0) -- (-3cm,0);
\begin{scope}[xshift=-11cm]
    \coordinate (P1) at (60:3cm);
    \coordinate (P2) at (0:3cm);
    \coordinate (P3) at (300:3cm);
    \coordinate (P4) at (240:3cm);
    \coordinate (P5) at (180:3cm);
    \coordinate (P6) at (120:3cm);

    \coordinate (A12) at (30:2.25cm);
    \coordinate (A23) at (330:2.25cm);
    \coordinate (A34) at (270:2.25cm);
    \coordinate (A45) at (210:2.25cm);
    \coordinate (A56) at (150:2.25cm);
    \coordinate (A61) at (90:2.25cm);

    \coordinate (L1) at ($(P1)!1.5!(A12)$);
    \coordinate (L2) at ($(P2)!1.5!(A23)$);
    \coordinate (L3) at ($(P3)!1.5!(A34)$);
    \coordinate (L4) at ($(P4)!1.5!(A45)$);
    \coordinate (L5) at ($(P5)!1.5!(A56)$);
    \coordinate (L6) at ($(P6)!1.5!(A61)$);

    \coordinate (R1) at ($(P1)!1.5!(A61)$);
    \coordinate (R2) at ($(P2)!1.5!(A12)$);
    \coordinate (R3) at ($(P3)!1.5!(A23)$);
    \coordinate (R4) at ($(P4)!1.5!(A34)$);
    \coordinate (R5) at ($(P5)!1.5!(A45)$);
    \coordinate (R6) at ($(P6)!1.5!(A56)$);

    \coordinate (Q1) at (-0.925cm,-1.61cm);
    \coordinate (Q2) at (0.795cm,-1.235cm);
    \coordinate (Q3) at (0.660cm,-0.47cm);
    \coordinate (Q4) at (0.286cm,-0.601cm);
    \coordinate (Q5) at (0.905cm,0.72cm);

    \coordinate (v1) at (barycentric cs:P1=1,P2=1,A12=1);
    \coordinate (v2) at (barycentric cs:P2=1,P3=1,A23=1);
    \coordinate (v3) at (barycentric cs:P3=1,P4=1,A34=1);
    \coordinate (v4) at (barycentric cs:P4=1,P5=1,A45=1);
    \coordinate (v5) at (barycentric cs:P5=1,P6=1,A56=1);
    \coordinate (v6) at (barycentric cs:P6=1,P1=1,A61=1);
    \coordinate (v7) at (barycentric cs:R1=1,L6=1,A61=1,Q3=1,Q4=1,Q5=1);
    \coordinate (v8) at (barycentric cs:R2=1,L1=1,A12=1,Q5=1);
    \coordinate (v9) at (barycentric cs:L5=1,A56=1,R6=1,L4=1,A45=1,R5=1,Q1=1);
    \coordinate (v10) at (barycentric cs:L3=1,A34=1,R4=1,Q1=1,Q2=1,Q4=1);
    \coordinate (v11) at (barycentric cs:L2=1,A23=1,R3=1,Q2=1,Q3=1);

\draw[fill=black, color=black] (v1) circle (1.75pt);
\draw[fill=black, color=black] (v2) circle (1.75pt);
\draw[fill=black, color=black] (v3) circle (1.75pt);
\draw[fill=black, color=black] (v4) circle (1.75pt);
\draw[fill=black, color=black] (v5) circle (1.75pt);
\draw[fill=black, color=black] (v6) circle (1.75pt);
\draw[fill=black, color=black] (v7) circle (1.75pt);
\draw[fill=black, color=black] (v8) circle (1.75pt);
\draw[fill=black, color=black] (v9) circle (1.75pt);
\draw[fill=black, color=black] (v10) circle (1.75pt);
\draw[fill=black, color=black] (v11) circle (1.75pt);

    \draw[color=black] (v1) -- (v2);
    \draw[color=black] (v2) -- (v3);
    \draw[color=black] (v3) -- (v4);
    \draw[color=black] (v4) -- (v5);
    \draw[color=black] (v5) -- (v6);
    \draw[color=black] (v6) -- (v1);
    \draw[color=black] (v1) -- (v8);
    \draw[color=black] (v8) -- (v7);
    \draw[color=black] (v6) -- (v7);
    \draw[color=black] (v10) -- (v7);
    \draw[color=black] (v11) -- (v7);
    \draw[color=black] (v11) -- (v10);
    \draw[color=black] (v11) -- (v2);
    \draw[color=black] (v10) -- (v3);
    \draw[color=black] (v9) -- (v5);
    \draw[color=black] (v9) -- (v4);
    \draw[color=black] (v9) -- (v10);
\end{scope}
\begin{scope}
    \coordinate (P1) at (60:3cm);
    \coordinate (P2) at (0:3cm);
    \coordinate (P3) at (300:3cm);
    \coordinate (P4) at (240:3cm);
    \coordinate (P5) at (180:3cm);
    \coordinate (P6) at (120:3cm);

    \coordinate (A12) at (30:2.25cm);
    \coordinate (A23) at (330:2.25cm);
    \coordinate (A34) at (270:2.25cm);
    \coordinate (A45) at (210:2.25cm);
    \coordinate (A56) at (150:2.25cm);
    \coordinate (A61) at (90:2.25cm);

    \coordinate (L1) at ($(P1)!1.5!(A12)$);
    \coordinate (L2) at ($(P2)!1.5!(A23)$);
    \coordinate (L3) at ($(P3)!1.5!(A34)$);
    \coordinate (L4) at ($(P4)!1.5!(A45)$);
    \coordinate (L5) at ($(P5)!1.5!(A56)$);
    \coordinate (L6) at ($(P6)!1.5!(A61)$);

    \coordinate (R1) at ($(P1)!1.5!(A61)$);
    \coordinate (R2) at ($(P2)!1.5!(A12)$);
    \coordinate (R3) at ($(P3)!1.5!(A23)$);
    \coordinate (R4) at ($(P4)!1.5!(A34)$);
    \coordinate (R5) at ($(P5)!1.5!(A45)$);
    \coordinate (R6) at ($(P6)!1.5!(A56)$);


    \coordinate (Q1) at (-0.925cm,-1.61cm);
    \coordinate (Q2) at (0.795cm,-1.235cm);
    \coordinate (Q3) at (0.660cm,-0.47cm);
    \coordinate (Q4) at (0.286cm,-0.601cm);
    \coordinate (Q5) at (0.905cm,0.72cm);

    \coordinate (v1) at (barycentric cs:P1=1,P2=1,A12=1);
    \coordinate (v2) at (barycentric cs:P2=1,P3=1,A23=1);
    \coordinate (v3) at (barycentric cs:P3=1,P4=1,A34=1);
    \coordinate (v4) at (barycentric cs:P4=1,P5=1,A45=1);
    \coordinate (v5) at (barycentric cs:P5=1,P6=1,A56=1);
    \coordinate (v6) at (barycentric cs:P6=1,P1=1,A61=1);
    \coordinate (v7) at (barycentric cs:R1=1,L6=1,A61=1,Q3=1,Q4=1,Q5=1);
    \coordinate (v8) at (barycentric cs:R2=1,L1=1,A12=1,Q5=1);
    \coordinate (v9) at (barycentric cs:L5=1,A56=1,R6=1,L4=1,A45=1,R5=1,Q1=1);
    \coordinate (v10) at (barycentric cs:L3=1,A34=1,R4=1,Q1=1,Q2=1,Q4=1);
    \coordinate (v11) at (barycentric cs:L2=1,A23=1,R3=1,Q2=1,Q3=1);

\draw[fill=black, color=black] (v1) circle (1.75pt);
\draw[fill=black, color=black] (v2) circle (1.75pt);
\draw[fill=black, color=black] (v3) circle (1.75pt);
\draw[fill=black, color=black] (v4) circle (1.75pt);
\draw[fill=black, color=black] (v5) circle (1.75pt);
\draw[fill=black, color=black] (v6) circle (1.75pt);
\draw[fill=black, color=black] (v7) circle (1.75pt);
\draw[fill=black, color=black] (v8) circle (1.75pt);
\draw[fill=black, color=black] (v9) circle (1.75pt);
\draw[fill=black, color=black] (v10) circle (1.75pt);
\draw[fill=black, color=black] (v11) circle (1.75pt);

\draw (v1) -- (v2);
\draw (v2) -- (v3);
\draw (v3) -- (v4);
\draw (v4) -- (v5);
\draw (v5) -- (v6);
\draw (v6) -- (v1);
    \draw[color=black] (v6) -- (v1);
    \draw[color=black] (v1) -- (v8);
    \draw[color=black] (v6) -- (v7);
    \draw[color=black] (v10) -- (v7);
    \draw[color=black] (v11) -- (v7);
    \draw[color=black] (v11) -- (v10);
    \draw[color=black] (v11) -- (v2);
    \draw[color=black] (v10) -- (v3);
    \draw[ color=black] (v9) -- (v5);
    \draw[color=black] (v9) -- (v4);
    \draw[color=black] (v9) -- (v10);
    \coordinate (n1) at ($(v8)!0.33!(v7)$);
    \draw[color=black] (v8) -- (n1);
\draw[fill=black, color=black] (n1) circle (1.75pt);

    \coordinate (n2) at ($(v8)!0.66!(v7)$);
    \draw[color=black] (n1) -- (n2);
\draw[fill=black, color=black] (n2) circle (1.75pt);

    \draw[color=black] (v7) -- (n2);

    \coordinate (m1) at ($(n1)+(0,0.4cm)$);
    \draw[color=black] (n1) -- (m1);
\draw[fill=black, color=black] (m1) circle (1.75pt);

    \coordinate (m2) at ($(n2)+(0,0.4cm)$);
    \draw[color=black] (n2) -- (m2);
\draw[fill=black, color=black] (m2) circle (1.75pt);

    \draw[color=black] (m2) -- (m1);
    \draw[color=black] (m2) -- (v7);
    \draw[color=black] (m1) -- (v8);

    \coordinate (s) at ($(n1)+(0,-0.4cm)$);
    \draw[color=black] (n1) -- (s);
\draw[fill=black, color=black] (s) circle (1.75pt);
    \draw[color=black] (s) -- (v7);
    \draw[color=black] (s) -- (v8);
\end{scope}
\end{tikzpicture}

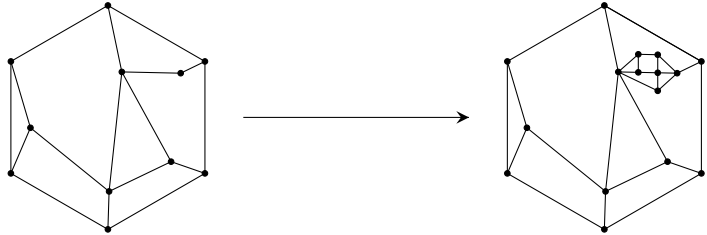
\captionof{figure}{The corresponding modification of $\mathcal{R}_b(G)$}
\end{center}
}\end{exmp}

So, we can simultaneously transform $\mathcal{R}_b(G)$ and $\mathcal{R}_w(G)$ to simple graphs such that the relation $\sim$ on elements of $\mathcal{O}$ is not changed. 
In particular, we come to a new ($4$-regular plane) graph $G$ and assert that $\mathcal{R}_b(G)$ contains a face for which $\sigma$ is realized as the $z$-monodromy of one of faces. 

Recall that $C$ is a cycle in $G$ with the vertices $p_1,\dots,p_k$ and it is the boundary of the outer face of $G$. 
This face has a common edge only with faces whose boundaries contain the vertices $p_i, p_j, a_{ij}$. 
By the definition of $\mathcal{R}_b(G)$, we have the following: 

$\bullet$ every face with the boundary containing $p_i, p_j, a_{ij}$ in $G$ is a vertex in $\mathcal{R}_b(G)$ which we denote by $v_{ij}$;

$\bullet$ the face bounded by $C$ in $G$ is a face in $\mathcal{R}_b(G)$ which will be denoted by $F$;

$\bullet$ every $p_i$ corresponds to an edge of $F$.

\noindent Consider the oriented edges $e_j=v_{ij}v_{jl}$ and $-e_j=v_{jl}v_{ij}$ in $\mathcal{R}_b(G)$, where $i,j,l$ are three consecutive elements in the cyclic sequence $1,\dots,k$. 
The pair of mutually reversed oriented edges $e_j, -e_j$ corresponds to the vertex $p_j$ in $G$. 
Thus, 
$$\Omega(F)=\{e_1,\dots,e_k,-e_k,\dots,-e_1\}.$$
Let $e_0,e\in\Omega(F)$ be such that $D_F(e_0)=e$. 
There is a unique zigzag $Z$ in $\mathcal{R}_b(G)$ containing the pair $e_0,e$. 
The element $e'$ which occurs in $Z$ directly after this pair does not belong to $\Omega(F)$. 
The edges $e_0,e,e'$ are three consecutive vertices in the central circuit in $G$ corresponding to $Z$ such that $e_0,e$ are two consecutive vertices from the cycle $C$ and $e'$ is one of elements $a_{ij}$. 
Let $x$ be the first element from $\mathcal{O}$ such that the central circuit containing $e_0,e,e'$ passes through $x$ (as a curve on the plane) directly after this triple ($x$ is a point on the plane, but not a vertex of the graph). 
There is a unique $x'\in\mathcal{O}$ such that $x\sim x'$ and the central circuit passes through $x'$. 
Since there is no elements of $\Omega(F)$ between $x$ and $x'$ in the central circuit, the first element of $\Omega(F)$ that occurs after $x'$ corresponds to $M_F(e)$. 
Therefore, $\sigma$ realizes as $M_F$. 

\begin{exmp}{\rm 
Let $F$ be the outer face of $\mathcal{R}_b(G)$ from Example \ref{ex5} and let $e_i$ be the oriented edge of $F$ corresponding to $p_i$ whose direction is defined by the clockwise orientation on the boundary of $F$ (see Fig. 13). 
\begin{center}
\begin{tikzpicture}[scale=0.9]
    \coordinate (P1) at (60:3cm);
    \coordinate (P2) at (0:3cm);
    \coordinate (P3) at (300:3cm);
    \coordinate (P4) at (240:3cm);
    \coordinate (P5) at (180:3cm);
    \coordinate (P6) at (120:3cm);

    \coordinate (A12) at (30:2.25cm);
    \coordinate (A23) at (330:2.25cm);
    \coordinate (A34) at (270:2.25cm);
    \coordinate (A45) at (210:2.25cm);
    \coordinate (A56) at (150:2.25cm);
    \coordinate (A61) at (90:2.25cm);

    \coordinate (L1) at ($(P1)!1.5!(A12)$);
    \coordinate (L2) at ($(P2)!1.5!(A23)$);
    \coordinate (L3) at ($(P3)!1.5!(A34)$);
    \coordinate (L4) at ($(P4)!1.5!(A45)$);
    \coordinate (L5) at ($(P5)!1.5!(A56)$);
    \coordinate (L6) at ($(P6)!1.5!(A61)$);

    \coordinate (R1) at ($(P1)!1.5!(A61)$);
    \coordinate (R2) at ($(P2)!1.5!(A12)$);
    \coordinate (R3) at ($(P3)!1.5!(A23)$);
    \coordinate (R4) at ($(P4)!1.5!(A34)$);
    \coordinate (R5) at ($(P5)!1.5!(A45)$);
    \coordinate (R6) at ($(P6)!1.5!(A56)$);


    \coordinate (Q1) at (-0.925cm,-1.61cm);
    \coordinate (Q2) at (0.795cm,-1.235cm);
    \coordinate (Q3) at (0.660cm,-0.47cm);
    \coordinate (Q4) at (0.286cm,-0.601cm);
    \coordinate (Q5) at (0.905cm,0.72cm);

    \coordinate (v1) at (barycentric cs:P1=1,P2=1,A12=1);
    \coordinate (v2) at (barycentric cs:P2=1,P3=1,A23=1);
    \coordinate (v3) at (barycentric cs:P3=1,P4=1,A34=1);
    \coordinate (v4) at (barycentric cs:P4=1,P5=1,A45=1);
    \coordinate (v5) at (barycentric cs:P5=1,P6=1,A56=1);
    \coordinate (v6) at (barycentric cs:P6=1,P1=1,A61=1);
    \coordinate (v7) at (barycentric cs:R1=1,L6=1,A61=1,Q3=1,Q4=1,Q5=1);
    \coordinate (v8) at (barycentric cs:R2=1,L1=1,A12=1,Q5=1);
    \coordinate (v9) at (barycentric cs:L5=1,A56=1,R6=1,L4=1,A45=1,R5=1,Q1=1);
    \coordinate (v10) at (barycentric cs:L3=1,A34=1,R4=1,Q1=1,Q2=1,Q4=1);
    \coordinate (v11) at (barycentric cs:L2=1,A23=1,R3=1,Q2=1,Q3=1);

\draw[fill=black, color=black] (v1) circle (1.5pt);
\draw[fill=black, color=black] (v2) circle (1.5pt);
\draw[fill=black, color=black] (v3) circle (1.5pt);
\draw[fill=black, color=black] (v4) circle (1.5pt);
\draw[fill=black, color=black] (v5) circle (1.5pt);
\draw[fill=black, color=black] (v6) circle (1.5pt);
\draw[fill=black, color=black] (v7) circle (1.5pt);
\draw[fill=black, color=black] (v8) circle (1.5pt);
\draw[fill=black, color=black] (v9) circle (1.5pt);
\draw[fill=black, color=black] (v10) circle (1.5pt);
\draw[fill=black, color=black] (v11) circle (1.5pt);

\draw [color=black, decoration={markings,
mark=at position 0.56 with {\arrow[scale=1.5,>=stealth]{>}}},
postaction={decorate}] (v1) -- (v2);
\draw [color=black, decoration={markings,
mark=at position 0.56 with {\arrow[scale=1.5,>=stealth]{>}}},
postaction={decorate}] (v2) -- (v3);
\draw [color=black, decoration={markings,
mark=at position 0.56 with {\arrow[scale=1.5,>=stealth]{>}}},
postaction={decorate}] (v3) -- (v4);
\draw [color=black, decoration={markings,
mark=at position 0.56 with {\arrow[scale=1.5,>=stealth]{>}}},
postaction={decorate}] (v4) -- (v5);
\draw [color=black, decoration={markings,
mark=at position 0.56 with {\arrow[scale=1.5,>=stealth]{>}}},
postaction={decorate}] (v5) -- (v6);
\draw [color=black, decoration={markings,
mark=at position 0.56 with {\arrow[scale=1.5,>=stealth]{>}}},
postaction={decorate}] (v6) -- (v1);
    \draw[color=black] (v6) -- (v1);
    \draw[color=black] (v1) -- (v8);
    \draw[color=black] (v6) -- (v7);
    \draw[color=black] (v10) -- (v7);
    \draw[color=black] (v11) -- (v7);
    \draw[color=black] (v11) -- (v10);
    \draw[color=black] (v11) -- (v2);
    \draw[color=black] (v10) -- (v3);
    \draw[ color=black] (v9) -- (v5);
    \draw[color=black] (v9) -- (v4);
    \draw[color=black] (v9) -- (v10);
    \coordinate (n1) at ($(v8)!0.33!(v7)$);
    \draw[color=black] (v8) -- (n1);
\draw[fill=black, color=black] (n1) circle (1.5pt);

    \coordinate (n2) at ($(v8)!0.66!(v7)$);
    \draw[color=black] (n1) -- (n2);
\draw[fill=black, color=black] (n2) circle (1.5pt);

    \draw[color=black] (v7) -- (n2);

    \coordinate (m1) at ($(n1)+(0,0.4cm)$);
    \draw[color=black] (n1) -- (m1);
\draw[fill=black, color=black] (m1) circle (1.5pt);

    \coordinate (m2) at ($(n2)+(0,0.4cm)$);
    \draw[color=black] (n2) -- (m2);
\draw[fill=black, color=black] (m2) circle (1.5pt);

    \draw[color=black] (m2) -- (m1);
    \draw[color=black] (m2) -- (v7);
    \draw[color=black] (m1) -- (v8);

    \coordinate (s) at ($(n1)+(0,-0.4cm)$);
    \draw[color=black] (n1) -- (s);
\draw[fill=black, color=black] (s) circle (1.5pt);
    \draw[color=black] (s) -- (v7);
    \draw[color=black] (s) -- (v8);

\node[color=black] at (60:2.5cm) {$e_1$};
\node[color=black] at (0:2.5cm) {$e_2$};
\node[color=black] at (300:2.5cm) {$e_3$};
\node[color=black] at (240:2.5cm) {$e_4$};
\node[color=black] at (180:2.5cm) {$e_5$};
\node[color=black] at (120:2.5cm) {$e_6$};

\end{tikzpicture}

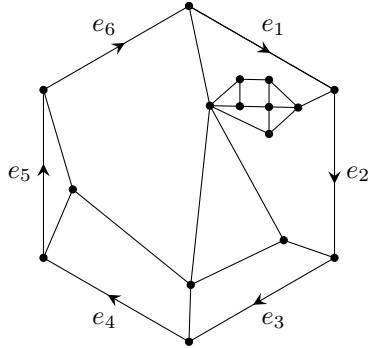
\captionof{figure}{The new graph $\mathcal{R}_b(G)$}
\end{center}
A direct verification shows that
$$M_F=(e_1,-e_6,-e_4,e_2)(e_3,-e_5)(e_5,-e_3)(-e_2,e_4,e_6,-e_1),$$
i.e. the permutation
$$\sigma=(1,-6,-4,2)(3,-5)(5,-3)(-2,4,6,-1)$$
from Example \ref{ex3} realizes as $M_F$ in $\mathcal{R}_b(G)$. 
}\end{exmp}

\section{The non-plane case}

In this section, we consider an arbitrary connected closed $2$-dimensional surface $S$ different from a sphere. 
We show that any permutation $\sigma$ on $[k]_{\pm}$ satisfying (M1) and (M2) realizes as the $z$-monodromy of $k$-gonal face in a graph embedded in $S$. 
Let $\Gamma$ be a graph embedded in a sphere (a plane graph) such that $\sigma$ realizes as the $z$-monodromy of a face $F$ of $\Gamma$. 
We assume that $\Gamma=\mathcal{R}_b(G)$, where $G$ is the $4$-regular graph from Section 4. 

Let $e$ be an edge in $G$.
It is contained in the boundaries of precisely two faces $F_1,F_2$ in $G$. 
We assume that $F_1$ and $F_2$ correspond to a face distinct from $F$ and a vertex of $\Gamma$, respectively. 
Let us take three circles $B_1, B_2, B_3$ that intersect like the Borromean rings. 
Consider the graph $G'$ obtained from $G$ by adding $B_1, B_2, B_3$ as in Fig. 14. 
\begin{center}
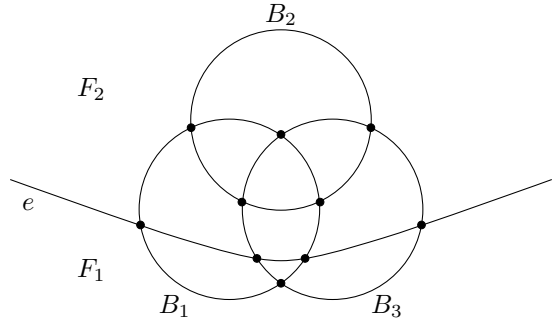

\begin{tikzpicture}[scale=1.2]


    \coordinate (B2) at (90:0.66cm);
    \coordinate (B1) at (210:0.66cm);
    \coordinate (B3) at (330:0.66cm);
\draw (B1) circle (1cm);
\draw (B2) circle (1cm);
\draw (B3) circle (1cm);

\draw[fill=black, color=black] (90:0.5cm) circle (1.2pt);
\draw[fill=black, color=black] (210:0.5cm) circle (1.2pt);
\draw[fill=black, color=black] (330:0.5cm) circle (1.2pt);

\draw[fill=black, color=black] (150:1.15cm) circle (1.2pt);
\draw[fill=black, color=black] (270:1.15cm) circle (1.2pt);
\draw[fill=black, color=black] (30:1.15cm) circle (1.2pt);

\node[color=black] at (90:1.83cm) {$B_2$};
\node[color=black] at (230:1.83cm) {$B_1$};
\node[color=black] at (310:1.83cm) {$B_3$};

\draw [black] plot [smooth] coordinates {(-3cm,0cm) (210:1.3cm) (270:0.9cm) (330:1.3cm) (3cm,0cm)};

\node[color=black] at (-2.8cm,-0.28cm) {$e$};

\draw[fill=black, color=black] ($(190:1cm)+(B1)$) circle (1.2pt);
\draw[fill=black, color=black] ($(213:1cm)+(B3)$) circle (1.2pt);
\draw[fill=black, color=black] ($(327:1cm)+(B1)$) circle (1.2pt);
\draw[fill=black, color=black] ($(350:1cm)+(B3)$) circle (1.2pt);

\node[color=black] at (-2.1cm,-1cm) {$F_1$};
\node[color=black] at (-2.1cm,1cm) {$F_2$};

\end{tikzpicture}
\captionof{figure}{Constructing of $G'$}
\end{center}
It must be pointed out that the circles $B_1, B_2, B_3$ do not intersect the remaining edges of $G$. 
The graph $G'$ is $4$-regular and $\mathcal{R}_b(G')$ is obtained from $\Gamma=\mathcal{R}_b(G)$ by adding the graph $\tilde{G}$ marked in red in Fig. 15 to the vertex $v$ corresponding to the face $F_2$. 
\begin{center}
\begin{tikzpicture}[scale=1.2]


    \coordinate (B2) at (90:0.66cm);
    \coordinate (B1) at (210:0.66cm);
    \coordinate (B3) at (330:0.66cm);
\draw (B1) circle (1cm);
\draw (B2) circle (1cm);
\draw (B3) circle (1cm);

\draw[fill=black, color=black] (90:0.5cm) circle (1.2pt);
\draw[fill=black, color=black] (210:0.5cm) circle (1.2pt);
\draw[fill=black, color=black] (330:0.5cm) circle (1.2pt);

\draw[fill=black, color=black] (150:1.15cm) circle (1.2pt);
\draw[fill=black, color=black] (270:1.15cm) circle (1.2pt);
\draw[fill=black, color=black] (30:1.15cm) circle (1.2pt);

\node[color=black] at (115:1.8cm) {$B_2$};
\node[color=black] at (230:1.83cm) {$B_1$};
\node[color=black] at (310:1.83cm) {$B_3$};

\draw [black] plot [smooth] coordinates {(-3cm,0cm) (210:1.3cm) (270:0.9cm) (330:1.3cm) (3cm,0cm)};

\node[color=black] at (-2.8cm,-0.28cm) {$e$};

\draw[fill=black, color=black] ($(190:1cm)+(B1)$) circle (1.2pt);
\draw[fill=black, color=black] ($(213:1cm)+(B3)$) circle (1.2pt);
\draw[fill=black, color=black] ($(327:1cm)+(B1)$) circle (1.2pt);
\draw[fill=black, color=black] ($(350:1cm)+(B3)$) circle (1.2pt);

    \coordinate (x1) at (90:2cm);
    \coordinate (x2) at (150:0.6cm);
\draw[fill=red, color=red] (x2) circle (1.75pt);
    \coordinate (x3) at (30:0.6cm);
\draw[fill=red, color=red] (x3) circle (1.75pt);

    \coordinate (x4) at (230:1.25cm);
\draw[fill=red, color=red] (x4) circle (1.75pt);
    \coordinate (x5) at (270:0.6cm);
\draw[fill=red, color=red] (x5) circle (1.75pt);
    \coordinate (x6) at (310:1.25cm);
\draw[fill=red, color=red] (x6) circle (1.75pt);

    \draw[color=red] (x1) -- (x2);
    \draw[color=red] (x1) -- (x3);
    \draw[color=red] (x2) -- (x3);
    \draw[color=red] (x5) -- (x2);
    \draw[color=red] (x5) -- (x3);
    \draw[color=red] (x5) -- (x4);
    \draw[color=red] (x5) -- (x6);
    \draw[color=red] (x4) -- (x6);
\draw [red] plot [smooth] coordinates {(x1) (150:1cm) (x4)};
\draw [red] plot [smooth] coordinates {(x1) (30:1cm) (x6)};

\node[color=black] at ($(0:0.4cm)+(x1)$) {$v$};

    \draw[color=black] (x1) -- ($(15:0.3cm)+(x1)$);
    \draw[color=black] (x1) -- ($(45:0.3cm)+(x1)$);
    \draw[color=black] (x1) -- ($(75:0.3cm)+(x1)$);
    \draw[color=black] (x1) -- ($(105:0.3cm)+(x1)$);
    \draw[color=black] (x1) -- ($(135:0.3cm)+(x1)$);
    \draw[color=black] (x1) -- ($(165:0.3cm)+(x1)$);

\node[color=black] at (0,0) {$T$};

\node[color=black] at (-2.1cm,-1cm) {$F_1$};
\node[color=black] at (-2.1cm,1cm) {$F_2$};

\draw[fill=red, color=red] (x1) circle (1.75pt);

\end{tikzpicture}

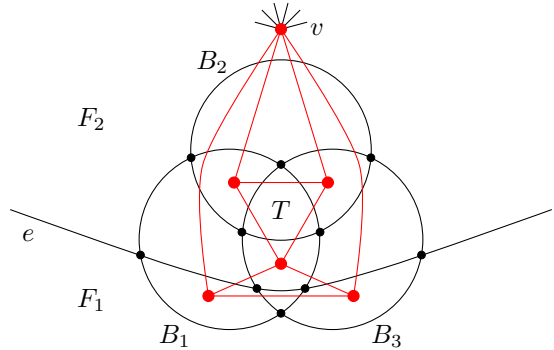
\captionof{figure}{The graph $\tilde{G}$}
\end{center}
It is clear that $\mathcal{R}_b(G')$ and $\mathcal{R}_w(G')$ are simple. 
Denote by $T$ the face of $\tilde{G}$ which is contained in $F_2$ and does not contain $v$. 
Note that $B_1, B_2, B_3$ induce central circuits of $G'$. 
Each zigzag of $\mathcal{R}_b(G')$ passing through $T$ corresponds to one of $B_i$. 
Observe that $F$ is the face of $\mathcal{R}_b(G')$ and the zigzags corresponding to $B_1, B_2, B_3$ do not contain edges of this face. 
This means that the $z$-monodromy of $F$ in $\mathcal{R}_b(G')$ is also $\sigma$. 

Consider any graph $\Gamma'$ embedded in $S$ that contains a triangle face $T'$. 
We take the connected sum of the sphere containing $\mathcal{R}_b(G')$ and $S$ by removing the interiors of faces $T$ and $T'$ and identifying theirs boundaries by a homeomorphism that sends vertices to vertices. 
We come to a new graph embedded in $S$ containing $F$ as a face. 
Since every zigzag of $\mathcal{R}_b(G')$ containing an edge of $F$ does not pass through any edge of $T$, the $z$-monodromy of $F$ in the new graph is the same as in $\mathcal{R}_b(G')$ and, consequently, as in $\mathcal{R}_b(G)$. 

\end{document}